\documentclass{article}
\usepackage{graphicx} 
\usepackage{tabularx}
\usepackage{booktabs}

\usepackage{amsmath,amssymb,amsthm,mathtools,mathrsfs}
\usepackage{xcolor}
\usepackage{caption}
\usepackage{graphicx}
\usepackage{mathtools}
\usepackage{caption}
\usepackage{floatrow}
\usepackage{tikz}
\usepackage{cite}
\usepackage{comment}
\usepackage{floatrow}
\usepackage{subcaption}
\usepackage{enumitem}
\usepackage{physics}
\usepackage[all,tips]{xy}
\usepackage{hyperref}
\usepackage{float}

\usepackage{array, booktabs}

\newcommand{\wtH}{\widetilde{H}}

\newtheorem{remark}{Remark}

\SelectTips{cm}{11}

\newtheorem{defn}{Definition}

\newtheorem{thm}{Theorem}

\newtheorem{cor}{Corollary}
\newtheorem{prop}{Proposition}

\newtheorem{theorem}{Theorem}
\newtheorem{corollary}{Corollary}
\newtheorem{example}{Example}

\newcommand\tran{\mkern-2mu\raise1.25ex\hbox{$\scriptscriptstyle\top\hspace{0.5mm}$}\mkern-3.5mu}

\usepackage[noabbrev]{cleveref} 
\crefname{rem}{Remark}{Remarks}
\crefname{exam}{Example}{Examples}
\crefname{assum}{Assumption}{Assumptions}
\crefname{algorithm}{Algorithm}{Algorithms}
\crefname{prop}{Proposition}{Propositions}
\crefname{propy}{Property}{Properties}
\crefname{cor}{Corollary}{Corollaries}
\crefname{lem}{Lemma}{Lemmas}
\crefname{section}{Section}{Sections}
\crefname{thm}{Theorem}{Theorems}
\crefname{defn}{Definition}{Definitions}
\crefname{figure}{Fig.}{Fig.}
\Crefname{figure}{Figure}{Figures}
\crefname{equation}{}{}

\title{Discrete-time generalized canonical transformations for non-autonomous systems}
\author{Leonardo Colombo$^1$, David Mart\'in de Diego$^2$,\\ Riccardo Muradore$^3$, Damiano Rigo$^4$, Nicola Sansonetto$^5$.\thanks{The authors acknowledge financial support from the Spanish Ministry of Science and Innovation under grants PID2022-137909NB-C21, PCI2024-155047-2,  Severo Ochoa Programme for Centres of Excellence in R\&D (CEX2023-001347-S) and by iRoboCity2030-CM, Robótica Inteligente para Ciudades Sostenibles (TEC-2024/TEC-62), funded by the Programas de Actividades I+D en Tecnologías en la Comunidad de Madrid.\\ 
$^1$Centre for Automation and Robotics (CSIC-UPM), Ctra. M300 Campo Real, Km 0,200, Arganda del Rey - 28500 Madrid, Spain, {\tt\small leonardo.colombo@csic.es},\\
$^2$Institute of Mathematical Sciences (CSIC-UAM-UCM-UC3M), Calle Nicolas Csbrera 13-15, Cantoblanco - Madrid, Spain, {\tt\small david.martin@icmat.es},\\
$^3$Department of Engineering for Innovation Medicine, University of Verona, Strada le Grazie 15, Verona, 37134, Italy {\tt\small riccardo.muradore@univr.it},\\
$^4$School of Mathematics, Jilin University, 2699 Qianjin Street, Changchun City, Jilin, 130012, China, {\tt\small damianorigo@jlu.edu.cn},\\
$^5$Department of Computer Science, University of Verona, Strada Le Grazie 15, Verona, Italy, {\tt\small nicola.sansonetto@univr.it}}}
\date{\empty}

\begin{document}

\maketitle

\begin{abstract}
A dynamical system is said to be \emph{non-autonomous} when the differential equations describing its evolution depends explicitly on time.  
Among the various geometric approaches to investigate such systems, the cosymplectic formulation provides a natural framework that extends symplectic geometry to time-dependent Hamiltonians systems.  
However, preserving the associated geometric structures under numerical discretization remains a challenging problem: standard integrators, such as explicit Euler schemes, generally fail to conserve the cosymplectic volume or the underlying Poisson structure.

In this work we propose a geometric method for the discretization of non-autonomous Hamiltonian systems based on \emph{generalized canonical transformations}.  
The approach constructs a symplectomorphism on the extended phase space $T^*(Q \times \mathbb{R})$ whose projection onto $T^*Q \times \mathbb{R}$ defines a structure-preserving discrete flow.  
We show that this formulation guarantees the preservation of key invariants, including the volume form, the Poisson bracket, and the symplectic structure on each time fiber.  
\end{abstract}

\section{Introduction}

The discretization of Hamiltonian systems has long been a central topic in geometric numerical integration \cite{sanz-serna,Hairer,blanes}. 
Symplectic integrators, in particular, are well known for preserving the underlying geometric structure of Hamiltonian dynamics, 
providing consistent tools for long-term stability and energy analysis.  
Their use across different areas of physics, control, and applied mathematics has motivated a broad class of structure-preserving schemes that retain essential invariants of continuous dynamics, such as volume and momentum conservation.

Although the geometric theory of discrete Hamiltonian systems is well established for \emph{autonomous} dynamics (see, e.g., \cite{kang, Hairer}), the case of \emph{non-autonomous} systems remains less systematically treated.
A possible way to study non-autonomous mechanical systems consists in considering the extended phase space $T^*(Q \times \mathbb{R})$, with local coordinates $(q,t, p, p_t)$, where the time is treated as variable with its own conjugate momentum.
In this setting, the dynamics is governed by the canonical extended  symplectic form on $T^*(Q\times \mathbb{R})$
\[
\Omega_0 = dq\wedge dp + dt\wedge dp_t,
\]
while the time-dependent space $T^*Q \times \mathbb{R}$ is endowed with a different geometry induced by the canonical \emph{cosymplectic} structure $(\omega_0,\eta) = (dq\wedge dp,\,dt)$.
Observe that now $\omega_0$ is a presymplectic 2-form since $\partial/\partial t$ is in $\ker \omega_0$.  
A central question then arises: \emph{can one construct numerical integrators that preserve the canonical geometric structures associated to this cosymplectic framework?}

Marthinsen and Owren in~\cite{marthinsen2016geometric} address this problem by characterizing discrete canonical transformations as symplectomorphisms on the extended space $T^*(Q \times \mathbb{R})$.  
Their analysis provides local matrix conditions that ensure that the extended symplectic form $\Omega_0$ is preserved by the discrete flow.  
However, their construction remains entirely within the symplectic setting and does not explicitly describe how such transformations act on the projected, time-dependent space $T^*Q \times \mathbb{R}$, 
where the relevant invariants are the cosymplectic volume and the Poisson bracket associated with $(\omega_0,\eta)$.

To better appreciate the role of these invariants, consider a time-dependent Hamiltonian $H:T^*Q \times \mathbb{R}\to\mathbb R$.
For sake of simplicity, take $Q=\mathbb R$ and $H(q,p)=\tfrac12e^{-t}(p^2+q^2)$, so that the continuous dynamics describes a harmonic oscillator with dissipation. 
The exact flow is given by
\[
\begin{pmatrix}q_{k+1}\\p_{k+1}\end{pmatrix}
=
\begin{pmatrix}
\cos\theta_h & \sin\theta_h \\
-\sin\theta_h & \cos\theta_h
\end{pmatrix}
\begin{pmatrix}
q_k \\
p_k
\end{pmatrix},
\]
where
$
\theta_h = e^{-t_k} - e^{-(t_k+h)},
$
and it preserves the canonical Poisson bracket or the symplecticity on fibers on the projection of $T^*Q \times \mathbb{R}$ to ${\mathbb R}$  since
\[
\begin{pmatrix}
\cos\theta_h & \sin\theta_h \\
-\sin\theta_h & \cos\theta_h
\end{pmatrix}^T\begin{pmatrix}
0 & 1 \\
-1 & 0
\end{pmatrix}
\begin{pmatrix}
\cos\theta_h & \sin\theta_h \\
-\sin\theta_h & \cos\theta_h
\end{pmatrix}=
\begin{pmatrix}
0 & 1 \\
-1 & 0
\end{pmatrix}
\]
However, if we use
an explicit Euler discretization on $T^*Q \times \mathbb{R}$:
\[
q_{k+1}=q_k+h\,\partial_p H(q_k, t_k, p_k),\quad
t_{k+1}=t_k+h,\quad
p_{k+1}=p_k-h\,\partial_q H(q_k, t_k, p_k),
\]
and we  denote by $\Phi_k:(q_k, t_k, p_k)\mapsto (q_{k+1}, t_{k+1}, p_{k+1})$ the one–step map defined by this explicit Euler update, then
\[
\begin{pmatrix}q_{k+1}\\p_{k+1}\end{pmatrix}
=
\begin{pmatrix}
1 & he^{-t_k} \\
- he^{-t_k} & 1
\end{pmatrix}
\begin{pmatrix}
q_k \\
p_k
\end{pmatrix},
\]
but
\[
\det \begin{pmatrix}
1 & he^{-t_k} \\
- he^{-t_k} & 1
\end{pmatrix}=1+h^2e^{-2t_k}>1
\]
and then the integrator is not symplectic. 
This simple example highlights the importance of enforcing symplecticity on each time slice $T^*Q \times \{t\}$ for the numerical methods:  
without ``slice-simplecticity'', even elementary oscillators exhibit spurious energy drift,  
whereas structure-preserving schemes maintain bounded error and reproduce the correct geometric features of the continuous flow.

The approach developed in this paper directly addresses this limitation.  
We construct discrete canonical transformations on the extended phase space $T^*(Q \times \mathbb{R})$ whose projections onto $T^*Q \times \mathbb{R}$ 
define structure-preserving integrators for non-autonomous Hamiltonian systems\cite{asorey1983generalized}.  
This construction relies on the notion of \emph{discretization maps}~\cite{barbero2023retraction}, 
which provide a geometric mechanism to approximate flows on tangent and cotangent bundles while preserving their intrinsic structures.  
By applying this formalism to the extended setting, we obtain a discrete flow that preserves both the volume $\omega_0^{\,n}\wedge\eta$ and the Poisson structure associated with the cosymplectic structure $(\omega_0,\eta)$, thereby providing a faithful discrete analogue of the continuous non-autonomous Hamiltonian dynamics.  
Crucially, for each fixed $t$, the induced map 
$\phi_{t,h}:T^*Q\to T^*Q$ is a symplectomorphism, ensuring per–time-slice symplecticity and geometric consistency of the discrete evolution (see \cite{campos} for a discrete Lagrangian description).

Our formulation extends the local results of \cite{marthinsen2016geometric} to a global, coordinate-free framework that connects symplectic and cosymplectic discretizations through the projection 
\(\mu:T^*(Q \times \mathbb{R})\to T^*Q \times \mathbb{R}\).  
This provides the first geometric proof of structure preservation for discrete canonical transformations in the non-autonomous case, 
bridging the discretization-map formalism of \cite{barbero2023retraction} with the canonical transformation theory of \cite{asorey1983generalized}.

\medskip
\noindent\textbf{Paper outline.}
Section \ref{section: Geometry of non-autonomous Hamiltonian systems} reviews the geometric background on extended phase spaces, cosymplectic structures, and Poisson morphisms.  
Section \ref{section: Discrete canonical transformations and structure-preserving integrators} introduces discrete canonical transformations generated by discretization maps and establishes their preservation properties.  
Section \ref{section: Numerical Simulations} presents numerical examples, including dissipative and mechanical systems with time-dependent potential, that illustrate the behaviour of the proposed integrators compared to standard Runge–Kutta schemes.
Section \ref{section: Application to an electric particle in an electromagnetic field} shows an example of how the system is well-suited for applications in gauge-invariant systems.
Finally, Section \ref{section: Conclusions and Future Work} summarizes the main conclusions and outlines future research directions.

All the functions, vector fields, forms and objects treated in this paper are assumed to be smooth unless explicitly stated.

\section{Geometry of non-autonomous Hamiltonian systems}\label{section: Geometry of non-autonomous Hamiltonian systems}

\subsection{Hamiltonian mechanics}
Let $Q$ be an $n$-dimensional configuration manifold of a mechanical system with local coordinates $(q^A)$, $1\leq A\leq n$ \cite{abtaham-marsden}.
If $T_{q}Q$ denotes the tangent space of $Q$ at the point $q$, the tangent bundle $TQ$ is defined by $\displaystyle{TQ:=\sqcup_{q\in Q}T_{q}Q}$, with local coordinates $(q^A, \dot{q}^A)$.
The bundle projection $\tau_{Q}: TQ \rightarrow Q$, sends each vector $v_{q}$ to the corresponding base point $q$, in local coordinates $\tau_{Q}(q^{A},\dot{q}^{A})=q^{A}$. 

$\displaystyle{T^{*}Q:=\sqcup_{q\in Q}T^{*}_{q}Q}$, endowed with local coordinates $(q^A,p_A)$, indicates the cotangent bundle.
$\pi_{Q}:T^{*}Q \rightarrow Q$ denotes the canonical projection and sends each momentum $p_{q}$ to the corresponding base point $q$.
In local coordinates $\pi_{Q}(q^{A},p_{A})=q^{A}$.



We associate a system of $2n$ first order ordinary differential equations to every function $H: T^{*}Q\to\mathbb{R}$, namely the associated Hamilton's equations
\begin{equation}\label{hameq1}
\dot{q}^{A}=\frac{\partial H}{\partial p_A},\quad\dot{p}_{A}=-\frac{\partial H}{\partial q^A},\quad 1\leq A\leq n.
\end{equation} 
Their solutions are integral curves of the Hamiltonian vector field $X_H$ associated with $H$ that locally reads $X_H(q,p)=\left(q, p, \frac{\partial H}{\partial p},-\frac{\partial H}{\partial q}\right)$.



A symplectic 2-form $\omega$ on a manifold $Q$ is a $(0,2)$-type tensor field that is skew-symmetric and non-degenerate, i.e., $\omega(X,Y)=-\omega(Y,X)$ for all vector fields $X$ and $Y$ and if $\omega(X,Y)=0$ for all vector fields $X$, then $Y\equiv 0$.
A pair $(Q,\omega)$ is called a symplectic manifold if $Q$ is a manifold and $\omega$ is a symplectic 2-form.

The set of vector fields and the set of 1-forms on $Q$ are denoted by $\mathfrak{X}(Q)$ and $\Omega^{1}(Q)$, respectively.

As described in~\cite{libermann2012symplectic}, the cotangent bundle $T^*Q$ of a manifold $Q$ is equipped with a canonical exact symplectic structure $\omega_Q=-d\theta_Q$, where $\theta_Q$ is the canonical 1-form on $T^*Q$. In canonical bundle coordinates $(q^A, p_A)$ on $T^*Q$, 
$
\theta_Q= p_A\, \mathrm{d}q^A$ and
$\omega_Q= \mathrm{d}q^A\wedge \mathrm{d}p_A.
$

The Hamiltonian dynamics is characterized by the following two essential properties~\cite{libermann2012symplectic}:
\begin{itemize}
\item the Hamiltonian function, hence the state of the system, is preserved along the dynamics:
$\omega_Q(X_H, X_H)=dH(X_H)=X_H(H)=0$;
\item the symplectic form $\omega$ is preserved along the dynamics, that is, 
 if $\Phi^t_{X_H}$ denotes the flow of $X_H$, then the pull-back of the symplectic form $\omega_Q$ by the flow is preserved, i.e., $(\Phi^t_{X_H})^*\omega_Q=\omega_Q$.

\end{itemize}

\begin{defn}
Let $(Q_1,\omega_1)$ and
$(Q_2,\omega_{2})$ be two symplectic manifolds, let $\phi:Q_1\to Q_2$ be a map. The
map $\phi$ is called symplectic if the symplectic forms are preserved:
$
\phi^*\omega_2=\omega_1$. Moreover, it is a \textit{symplectomorphism} if additionally $\phi$ is a diffeomorphism.
\end{defn}

Let $Q_1$ and $Q_2$ be $n$-dimensional manifolds and $F: Q_1\rightarrow Q_2$ be a map.
The {\it tangent lift} $TF: TQ_1\rightarrow TQ_2$  of $F$ is defined by $TF(v_q)=T_qF (v_q)\in T_{F(q)} Q_2\, \, \mbox{ where } v_q\in T_qQ_1$, and $T_qF$ is the tangent map of $F$ whose matrix representation is the Jacobian matrix of $F$ at $q\in Q_1$.

As the tangent map $T_qF$ is linear, the dual map $T_{q}^*F\colon T^*_{F(q)}Q_2\rightarrow T^*_qQ_1$ is defined as follows:
\[\langle(T^*_{q}F)(\alpha_2), v_{q}\rangle=\langle \alpha_2, T_{q}F(v_{q})\rangle\mbox{ for every } v_q\in T_qQ_1.\]
Note that $(T^*_{q}F)(\alpha_2)\in T^*_qQ_1$. 
  
\begin{defn}\label{def:colift}
Let $F: Q_1\rightarrow Q_2$ be a diffeomorphism.
The vector bundle morphism $\widehat{F}: T^*Q_1\rightarrow T^*Q_2$ defined by	$\widehat{F}=T^*F^{-1}$ is called the cotangent lift of $F^{-1}$.
In other words, $\widehat{F}(\alpha_q)= 	T^*_{F(q)}F^{-1} (\alpha_q)$ where $\alpha_q\in T^*_q Q_1$.
Note that, $(T^*F^{-1})\circ (T^*F)={\rm Id}_{T^*Q_2}$.
\end{defn}

\subsection{Time-dependent Hamiltonian systems and cosymplectic geometry} 





To describe the dynamics of a non-autonomous Hamiltonian system, we consider the manifold $T^*Q \times \mathbb{R}$ equipped with 
$(\omega_0, \eta)$,
where 
$\omega_0=\hbox{pr}_1^*\omega_Q$ and $\eta=\hbox{pr}_2^* dt$.
Here $\hbox{pr}_1: T^*Q \times \mathbb{R} \rightarrow T^*Q$ and $\hbox{pr}_2: T^*Q \times \mathbb{R} \rightarrow \mathbb{R}$ denote the projections.

Note that the 2-form $\omega_0$ is \emph{presymplectic}, since its kernel is generated by $\partial_t$. 
This degeneracy reflects the fact that the time variable $t$ plays the role of an external parameter rather than a canonical coordinate. 
Consequently, the pair $(\omega_0, \eta)$ provides the canonical \emph{cosymplectic structure} on $T^*Q \times \mathbb{R}$, 
which constitutes the natural geometric framework for time-dependent Hamiltonian systems (see \cite{cappelletti} and references therein for details).

Since $\omega_0^n\wedge dt\not= 0$, then 
$(T^*Q \times \mathbb{R}, \omega_0, \eta)$ is a cosymplectic manifold. 
In Darboux coordinates $(q^i, t, p_i)$, we have that
\[
\omega_0=dq^i\wedge dp_i, \qquad \eta=dt.
\]

The cosymplectic structure defines a unique vector field $R$, the Reeb vector field, that satisfies
\[
i_R\omega_0=0, \qquad i_R \eta=1.
\]
In the particular case studied in this work, the Reeb vector field $R=\partial/\partial t=\partial_t$.


For any function $f\in C^{\infty}(T^*Q \times \mathbb{R})$ we define the {\it Hamiltonian} vector field $X_f$ by
\[
i_{X_f}\omega_0=df-R(f)\eta, \qquad i_{X_f}\eta=0,
\]
and the {\it evolution vector field} $E_f$ by $E_f=R+X_f$ or equivalently
\begin{equation*}\label{def:evolution}
i_{E_f}\omega_0=df-R(f)\eta, \qquad i_{E_f}\eta=1\;.
\end{equation*}
The Lie derivative of $\omega_0$ and $\eta$ along $E_f$ produces
\begin{equation*}
\mathcal{L}_{E_f}\omega_0=-d(R(f))\wedge \eta, \qquad \mathcal{L}_{E_f}\eta=0.
\end{equation*}

Locally, the evolution vector field of the Hamiltonian function
$H$ reads 
\[
E_H=R+X_H=\frac{\partial}{\partial t}+\frac{\partial H}{\partial p_A}\frac{\partial }{\partial q^A}-\frac{\partial H}{\partial q^A}\frac{\partial }{\partial p_A}.
\]
Its integral curves coincide with the trajectories defined by Hamilton’s equations
\[
\dot q=\partial_p H, \quad \dot t = 1, \quad \dot p=-\partial_q H.
\]

The cosymplectic structure $(\omega_0, \eta)$ induces a Poisson bracket $\{\cdot,\cdot\}_0$ on $T^*Q \times \mathbb{R}$ 
by
\[
\{f, g\}_0=\omega_0(X_f, X_g)=\omega_0(E_f, E_g)
\]
for all $f, g\in C^{\infty}(T^*Q \times \mathbb{R})$, that in Darboux coordinates is given by the bivector
\[
\Lambda_0=\frac{\partial}{\partial q^A}\wedge \frac{\partial}{\partial p_A}\; .
\]
Observe that $\mathcal{L}_{E_H}\Lambda_0=\mathcal{L}_{R}\Lambda_0+\mathcal{L}_{X_H}\Lambda_0=0$ since $\mathcal{L}_R\Lambda_0=0$.
The canonical volume form on $T^*Q \times \mathbb R$ $\mathrm{vol}_0=\omega_0^n\wedge \eta$ is preserved by the flow of $E_H$ as well.

An interesting way to represent the dynamics of the evolution vector field $E_H$ is to consider the modified cosymplectic structure  by taking the cosymplectic pair $(\Omega_H, dt)$ where 
\[
\Omega_H=\omega_0+dH\wedge \eta,\quad \eta=dt,
\]
with corresponding Reeb vector field precisely the evolution vector field:
\[
i_{E_H}\Omega_H=0,\quad i_{E_H}\eta=1. 
\]
Therefore, the flow of $E_H$ satisfies
\[
\mathcal{L}_{E_H}\Omega_H = 0,
\quad
\mathcal{L}_{E_H}dt = 0.
\]

We stress the fact that this property is difficult to preserve using numerical methods, since $\Omega_H$ depends on the Hamiltonian function $H$.
Therefore, we aim to obtain infinitesimal preservation properties in terms of canonical objects.

We also observe that the Evolution vector field $E_H$ preserves:
\begin{enumerate}
    \item the volume $\mathrm{vol}_0=\omega_0^n\wedge \eta$, that is $\mathcal{L}_{E_H}\mathrm{vol}_0=0$; 
    \item the Poisson structure $\Lambda$, that is $\mathcal{L}_{E_H}\Lambda_0=0$.
\end{enumerate}

Denote now by $
\Psi^H_s\colon U\subset T^*Q \times \mathbb{R}\rightarrow  T^*Q \times \mathbb{R}
$
the flow of the evolution vector field $E_H$, where $U$ is an open subset of $T^*Q \times \mathbb{R}$.
Observe that if $\alpha_q\in T^*_qQ$
\[
  \Psi^H_s(\alpha_q, t) = (\Psi^H_{t,s}(\alpha_q), t+s) \,,
\]
where $\Psi^H_{t,s}(\alpha_q)={pr}_1(\Psi^H_s(\alpha_q, t))$.
Therefore, we induce a map
\[
\Psi^H_{t,s}\colon U_t\subseteq T^*Q\to T^*Q\,,
\]
where $U_t=\{ \alpha_q\in T^*Q \;|\; (t, \alpha_q)\in U \}$. 
In terms of the flow of $E_H$ we have then that the flow of the evolution vector field $E_H$ preserves:
\begin{enumerate}
    \item volume: $(\Psi^H_s)^*\mathrm{vol}_0=
    \mathrm{vol}_0$;
    \item Poisson bracket: $\{f\circ \Psi^H_s, g\circ \Psi^H_s\}_0=\{f, g\}_0\circ \Psi^H_s$;
    \item Symplectic form: $(\Psi^H_{t,s})^*\omega_Q=\omega_Q$.
\end{enumerate}
These are the three related properties that are important to preserve when constructing appropriate geometric integrators.

\subsection{Generalized canonical transformations}

In order to ensure geometric preservations, we will use a symplectification method  to work in the standard symplectic framework (see \cite{asorey1983generalized} and references therein for details).
Given the Hamiltonian function $H:T^*Q \times \mathbb{R}\rightarrow \mathbb{R}$ we can construct the extended Hamiltonian 
$\tilde{H}:T^*(Q \times \mathbb{R})\rightarrow \mathbb{R}$ by taking 
\[
\tilde{H}=H\circ \mu+p_t
\]
where $\mu: T^*
(Q \times \mathbb{R})\rightarrow T^*Q \times \mathbb{R}$ is the projection
$\mu( \alpha_q; t, p_t)=( \alpha_q, t)$,
with $(t, p_t)\in T^*{\mathbb R}=\mathbb{R}\times {\mathbb R}^*$ and $\alpha_q\in T^*_q Q$.
In Darboux coordinates $(q^A,t,p_A,p_t)$ on $T^*(Q \times \mathbb{R})$,
the map $\mu$ reads
\[
\mu(q^A,t,p_A,p_t) = (q^A,p_A,t) \in T^*Q \times \mathbb{R}.
\]

$(T^*(Q \times \mathbb{R}), \omega_{Q \times \mathbb{R}})$  is a symplectic manifold with $\omega_{Q \times \mathbb{R}}$ the canonical symplectic form, that in Darboux coordinates reads
\[
\omega_{Q \times \mathbb{R}} = dq^A\wedge dp_A + dt\wedge dp_t.
\]

\begin{theorem}[~\cite{asorey1983generalized}]
The Hamiltonian vector field $X_{\tilde{H}}$ associated to $\tilde{H}$ is obtained as
\[
i_{X_{\tilde{H}}}\omega_{Q \times \mathbb{R}}=d\tilde{H}
\]
and is $\mu$-related to the evolution vector field, that is, $T\mu (X_{\tilde{H}})=E_H$.
\end{theorem}

The proof is quite simple since the expression of $X_{\tilde{H}}$ in coordinates is: 
\begin{align*}
X_{\tilde{H}}&=\frac{\partial \tilde{H}}{\partial p_A}\frac{\partial}{\partial q^A}-\frac{\partial \tilde{H}}{\partial q^A}\frac{\partial}{\partial p_i}
+\frac{\partial \tilde{H}}{\partial p_t}\frac{\partial}{\partial t}
-\frac{\partial \tilde{H}}{\partial t}\frac{\partial}{\partial p_t}\\
&=\frac{\partial}{\partial t}+\frac{\partial ({H}\circ \mu)}{\partial p_A}\frac{\partial}{\partial q^A}-\frac{\partial (H\circ \mu)}{\partial q^A}\frac{\partial}{\partial p_A}
-\frac{\partial (H\circ \mu)}{\partial t}\frac{\partial}{\partial p_t}.
\end{align*}
In $T^*(Q \times \mathbb{R})$ the associated canonical bracket $\{\, ,\,\}$ 
is given by
\[
\{\tilde{f}, \tilde{g}\}=\omega_{Q \times \mathbb{R}}(X_{\tilde{f}}, X_{\tilde{g}})
\]
for all $\tilde{f}, \tilde{g}\in C^{\infty}(T^*(Q \times \mathbb{R}))$.
In Darboux coordinates is represented by the bivector
\[
\Lambda=\frac{\partial}{\partial q^A}\wedge \frac{\partial}{\partial p_A}+
\frac{\partial}{\partial t}\wedge \frac{\partial}{\partial p_t}
\; .
\]
Additionally, as an immediate consequence of the definition of the Poisson structures on $T^*(Q \times \mathbb{R})$ and $T^*Q \times \mathbb{R}$ it emerges that $\mu: T^*(Q \times \mathbb{R})\rightarrow T^*Q \times \mathbb{R}$  is a Poisson map, that is,
\[
\{\mu^*f, \mu^*g\}=\mu^*\{f, g\}_0
\]
for all $f, g\in C^{\infty}(T^*Q \times \mathbb{R})$.



The concept of canonical transformation does not apply directly to time-dependent mechanical systems in $T^*Q \times \mathbb{R}$, but a similar notion can be formulated considering the extended phase space.

\begin{defn} [~\cite{asorey1983generalized, marthinsen2016geometric}]
    A canonical transformation of a time-dependent system $(T^*Q \times \mathbb{R},\omega_0,\eta)$ is a pair $(\psi, \varphi)$ of diffeomorphisms, $\psi$ on $T^*(Q \times \mathbb{R})$ and $\varphi$ on $T^*Q \times \mathbb{R}$ such that
    \begin{enumerate}
        \item $\mu \circ \psi = \varphi \circ \mu$;
        \item $\psi$ is a symplectomorphism of $(T^*(Q \times \mathbb{R}), \omega_{Q \times \mathbb{R}})$.
    \end{enumerate}
\end{defn}
This definition requests that the diagram in Figure \ref{Diag: canonical transformation} commutes.
\begin{figure} [htb!]
 \begin{equation*}
\xymatrix{ 
{{ T^*(Q \times \mathbb{R}) }} \ar[rr]^{{{\psi}}} \ar[d]_{\mu} && \ar[d]_{\mu}
{{T^*(Q \times \mathbb{R}) }}  \\ 
T^*Q \times \mathbb{R} \ar[rr]^{\varphi}&& T^*Q \times \mathbb{R} }
\end{equation*}
    \caption{Scheme for canonical transformation.}\label{Diag: canonical transformation}
\end{figure}

It is important to notice that this definition does not depend on the Hamiltonian function $H$.

\begin{prop}\label{Propo4}
If $(\psi,\varphi)$ is a canonical transformation, then the map $\varphi$ preserves  the Poisson structure on $(T^*Q \times \mathbb{R},\omega_0,\eta)$, that is $\{f\circ\varphi,\,g\circ\varphi\}_0=\{f,g\}_0\circ\varphi$ for all $f,g\in C^{\infty}(T^*Q \times \mathbb{R})$.
\end{prop}

\begin{proof}
Since $\mu$ is a Poisson morphism
\[
\mu^{*}\{f,g\}_0=\{\mu^{*}f,\mu^{*}g\},\qquad 
\forall f,g\in C^{\infty}(T^*Q \times \mathbb{R}),
\]
where $\{\cdot,\cdot\}$ is the canonical Poisson bracket on $T^*(Q \times \mathbb{R})$.  
Since $\psi$ is symplectic, it is also a Poisson map:
\[
\{\psi^{*}F,\psi^{*}G\}=\psi^{*}\{F,G\},\qquad \forall F,G\in C^{\infty}(T^*(Q \times \mathbb{R})).
\]
Then, for any $f,g$,
\[
\begin{aligned}
\mu^{*}\{f\circ\varphi,\,g\circ\varphi\}_0
&=\{\mu^{*}(f\circ\varphi),\,\mu^{*}(g\circ\varphi)\}
 =\{f\circ(\varphi\circ\mu),\,g\circ(\varphi\circ\mu)\}\\
&=\{f\circ(\mu\circ\psi),\,g\circ(\mu\circ\psi)\}
 =\{\psi^{*}(\mu^{*}f),\,\psi^{*}(\mu^{*}g)\}\\
&=\psi^{*}\{\mu^{*}f,\,\mu^{*}g\}
 =\psi^{*}(\mu^{*}\{f,g\}_0)
 =(\mu\circ\psi)^{*}\{f,g\}_0\\
& =\mu^{*}(\{f,g\}_0\circ\varphi).
\end{aligned}
\]
Because $\mu$ is surjective,
\[
\{f\circ\varphi,\,g\circ\varphi\}_0=\{f,g\}_0\circ\varphi.
\]
Thus $\varphi$ preserves the Poisson structure $\{\cdot,\cdot\}_0$.
\end{proof}

Now, consider $\tau_s(q,p,t,p_t)=(q,p,t,p_t+s)$ the natural translation by $s\in {\mathbb R}$ along the fibers.
Then we have the following proposition.

\begin{prop}\label{Prop5}
If $(\psi,\varphi)$ is a canonical transformation and if
$\psi\circ \tau_s=\tau_s\circ \psi$,
then the map $\varphi$ preserves 
the volume $\mathrm{vol}_0=\omega_0^{\,n}\wedge\eta$ in $(T^*Q \times \mathbb{R},\omega_0,\eta)$, that is, $\varphi^{*}(\mathrm{vol}_0)=\mathrm{vol}_0$.
%
%
\end{prop}

\begin{proof}

Denote by $\mathrm{Vol}=\omega_{Q \times \mathbb{R}}^{\,n+1}$ the Liouville volume on $T^*(Q \times \mathbb{R})$
and by $\mathrm{vol}_0=\omega_0^{\,n}\wedge\eta$ the canonical cosymplectic volume on $T^*Q \times \mathbb{R}$. 

Because $\psi$ is symplectic, it preserves the Liouville volume $\psi^{*}\mathrm{Vol}=\mathrm{Vol}$ but, furthermore, since it is a canonical transformation verifying the additional condition that 
$\psi\circ \tau_s=\tau_s\circ \psi$, then we have that  
$\psi^{*}(\mu^{*}\mathrm{vol}_0)=\mu^{*}\mathrm{vol}_0 $.

Using $\mu\circ\psi=\varphi\circ\mu$, we compute:
\[
\mu^{*}\big(\varphi^{*}\mathrm{vol}_0\big)
=(\varphi\circ\mu)^{*}\mathrm{vol}_0
=(\mu\circ\psi)^{*}\mathrm{vol}_0
=\psi^{*}(\mu^{*}\mathrm{vol}_0)
=\mu^{*}\mathrm{vol}_0.
\]
Since $\mu$ is a surjective submersion,
\[
\varphi^{*}\mathrm{vol}_0=\mathrm{vol}_0.
\]
Hence $\varphi$ preserves the volume form $\omega_0^{\,n}\wedge\eta$.

\end{proof}


\medskip


Observe that the flow of the Hamiltonian $\tilde{H}=H\circ \mu+p_t$ verifies that $(\tau_s)_* X_{\tilde{H}}=X_{\tilde{H}}$, therefore its flow commutes with the translation along the fibers. 
Now, our aim is to construct symplectic integrators that are canonical transformations and verify the commuting condition with respect to $\tau_s$ obtaining preservation of Poisson bracket and volume form by applying Propositions \ref{Propo4} and \ref{Prop5}.

\section{Discrete canonical transformations and structure-preserving integrators}\label{section: Discrete canonical transformations and structure-preserving integrators}

The first notion of retraction appearing in the literature can be found in~\cite{borsuk1947topology, borsuk1931symmetric} from a topological viewpoint. Later on, the notion of retraction map as defined below is used to obtain Newton's method on Riemannian manifolds~\cite{shub1986some,adler2002newton}.

\begin{defn}\label{def-RetractMap} A \textit{retraction map} on a manifold $Q$ is a smooth mapping $R$ from the tangent bundle $TQ$ onto $Q$. Let $R_q$ denote the restriction of $R$ to $T_qQ$, the following properties are satisfied:
\begin{enumerate}
	\item $R_q(0_q)=q$, where $0_q$ denotes the zero element of the vector space $T_qQ$.
	\item With the canonical identification $T_{0_q}T_qQ\simeq T_qQ$, $R_q$ satisfies \begin{equation}\label{eq-DefRetract-prop}
	DR_q(0_q)=T_{0_q}R_q={\rm Id}_{T_qQ},
	\end{equation}
	where ${\rm Id}_{T_qQ}$ denotes the identity mapping on $T_qQ$.
\end{enumerate}
\end{defn}

The condition~\eqref{eq-DefRetract-prop} is known as \textit{local rigidity condition} since, given $\xi\in T_qQ$,  the curve $\gamma_\xi(t)=R_q(t\xi)$ has $\xi$ as tangent vector at $q$, i.e. $\dot{\gamma}_\xi(t)= \langle DR_q(t\xi), \xi\rangle\; \hbox{ and, in consequence}, \dot{\gamma}_\xi(0)= {\rm Id}_{T_qQ}(\xi)=\xi$.

A typical example of a retraction map is the exponential map, $\hbox{exp}$, on Riemannian manifolds given for example in~\cite[Chapter 3.2]{do1992riemannian}. Therefore, the image of $\xi$ through the exponential map is a point on the Riemannian manifold $(Q,g)$ obtained by moving along a geodesic a length equal to the norm of $\xi$ starting with the velocity $\xi/\|\xi\|$, that is, $\hbox{exp}_q(\xi)=\sigma (\|\xi\|),$ where $\sigma$ is the unit speed geodesic such that $\sigma(0)=q$ and $\dot{\sigma}(0)=\xi/\|\xi\|$. 

Next, we define a generalization of the retraction map in Definition~\ref{def-RetractMap} that allows a discretization of the tangent bundle of the configuration manifold leading to the construction of numerical integrators as described in \cite{barbero2023retraction}. Given a point and a velocity, we obtain two nearby points that are not necessarily equal to the initial base point. 

\begin{defn} \label{def:DiscreteMap2} A map 
	$R_d\colon U\subset TQ\rightarrow Q\times Q$ given by \begin{equation*}
	R_d(q,v)=(R^1(q,v),R^2(q,v)),
	\end{equation*} 
	where  $U$ is an open neighborhood of the zero section $0_q$ of $TQ$, 
	defines a {\it discretization map on $Q$} if it satisfies 
	\begin{enumerate}
		\item $R_d(q,0)=(q,q)$,
		\item $T_{0_q}R^2_q-T_{0_q}R^1_q\colon T_{0_q}T_qQ\simeq T_qQ\rightarrow T_qQ$ is equal to the identity map on $T_qQ$ for any $q$ in $Q$, where $R^a_q$ denotes the restrictions of $R^a$, $a=1,2$, to $T_qQ$.
	\end{enumerate}
\end{defn}
Thus, the discretization map $R_d$ is a local diffeomorphism from some neighborhood of the zero section of $TQ$.

If $R^1(q,v)=q$, the two properties in Definition~\ref{def:DiscreteMap2} guarantee that the both properties in Definition~\ref{def-RetractMap} are satisfied by $R^2$. Thus, Definition~\ref{def:DiscreteMap2} generalizes Definition~\ref{def-RetractMap}. 

		

\begin{example} The mid-point rule on an Euclidean vector space can be recovered from the following discretization map:
$R_d(q,v)=\left( q-\dfrac{v}{2}, q+\dfrac{v}{2}\right).$

Examples of retraction maps on Euclidean vector spaces typically used in the literature for the construction of numerical methods (see \cite{iserles2009first} and \cite{hairer2006geometric} for instance), that can be used to define discretization maps are:
\begin{itemize}
	\item Explicit Euler method:  $R_d(q,v)=(q,q+v).$
	\item Midpoint rule:  $R_d(q,v)=\left( q-\dfrac{v}{2}, q+\dfrac{v}{2}\right).$
		\item $\theta$-methods: $R_d(q,v)=\left( q-\theta \, v, q+ (1-\theta)\, v\right),$ with $\theta\in [0,1]$.
\end{itemize}
\end{example}


\subsection{Cotangent lift of discretization maps}

As the Hamiltonian vector field takes value on $TT^*Q$, the discretization map must be on $T^*Q$, that is, 
$R^{T^*}_d: TT^*Q \rightarrow T^*Q\times T^*Q$. Such a map is obtained by cotangently lifting a discretization map $R_d\colon TQ\rightarrow Q\times Q$, so that the construction $R^{T^*}_d$ is a symplectomorphism. In order to do that, we need the following three symplectomorphisms (see \cite{barbero2023retraction} for more details): 
\begin{itemize}
\item The cotangent lift of the diffeomorphism $R_d\colon TQ\rightarrow Q\times Q$ as described in Definition~\ref{def:colift}. 
	\item The canonical symplectomorphism $\alpha_Q\colon TT^*Q  \longrightarrow  T^*TQ$  
	 such that\newline $\alpha_Q(q,p,v_q,v_p)= (q,v_q,v_p,p)$.

	\item  The symplectomorphism between $(T^*(Q\times Q), \omega_{Q\times Q})$     and   
	$(T^*Q\times T^*Q, \Omega_{12}:=pr_2^*\omega_Q-pr^*_1\omega_Q)$:
	$\Phi:T^*Q\times T^*Q \longrightarrow T^*(Q\times Q)$, given by $\Phi(q_0, p_0; q_1, p_1)=(q_0, q_1, -p_0, p_1).$
	\end{itemize}
Diagram in Figure~\ref{Diag:RdT*} summarizes the construction process from $R_d$ to $R_d^{T^*}$.
	\begin{figure} [htb!]
 \begin{equation*}
\xymatrix{ 
{TT^*M } \ar[rr]^{R_d^{T^*}} \ar[d]^{\alpha_{M}} && {T^*M\times T^*M } \ar[d]^{\Phi}
\\ 
T^*TM \ar[d]^{\pi_{TM}} \ar[rr]^{ \widehat{R_d}} && T^*(M\times M) \ar[d]^{\pi_{M\times M}}
\\ 
TM \ar[rr]^{R_d} && M\times M 
}
\end{equation*}
    \caption{Definition of the cotangent lift of a discretization.}\label{Diag:RdT*}
\end{figure}

\begin{prop} [see \cite{barbero2023retraction}]
	Let $R_d\colon TQ\rightarrow Q\times Q$ be a discretization map on $Q$. Then $${{R_d^{T^*}=\Phi^{-1}\circ \widehat{R_d}\circ \alpha_Q\colon TT^*Q\rightarrow T^*Q\times T^*Q}}$$ 
is a discretization map  on $T^*Q$.\label{Prop:RdT*}
\end{prop}

\begin{corollary} [see \cite{barbero2023retraction}]
The discretization map
 ${{R_d^{T^*}}} = \Phi^{-1}\circ (TR_d^{-1})^* \circ \alpha_Q\colon T(T^*Q) \rightarrow T^*Q\times T^*Q$ is a symplectomorphism between $(T(T^*Q), {\rm d}_T \omega_Q)$ and $(T^*Q\times T^*Q, \Omega_{12})$, where ${\rm d}_T \omega_Q = dq \wedge d\dot{p} + d\dot{q} \wedge dp$.
\end{corollary}
\begin{example}\label{example3} On $Q={\mathbb R}^n$ the discretization map 
	$R_d(q,v)=\left(q-\frac{1}{2}v, q+\frac{1}{2}v\right)$ is cotangently lifted to
		$$R_d^{T^*}(q,p,\dot{q},\dot{p})=\left( q-\dfrac{1}{2}\,\dot{q}, p-\dfrac{\dot{p}}{2}; \; q+\dfrac{1}{2}\, \dot{q}, p+\dfrac{\dot{p}}{2}\right)\, .$$
\end{example}

\subsection{Geometric integrators for non-autonomous systems}\label{section: Geometric integrators for non-autonomous systems}

The aim of this section is to extend the construction of geometric integrators to the non-autonomous Hamiltonian setting.  
Our goal is to construct a discrete analogue of the continuous Hamiltonian flow that preserves the same geometric structures - namely, a \emph{canonical transformation} on the extended phase space and its projection onto the cosymplectic manifold $T^*Q \times \mathbb{R}$.

Given a smooth manifold $M$, consider a discretization map $R_d : TM \longrightarrow M\times M$, which associates to each tangent vector a discrete displacement between two nearby configurations.  
The differential of $R_d$ allows us to define its cotangent lift $\widehat{R}_d:T^*TM\longrightarrow T^*(M\times M)$, which provides the natural bridge between discrete Lagrangian and Hamiltonian formulations.

A key ingredient in the construction is the canonical \emph{anti-symplectomorphism}
\[
{\mathcal I}_{TM}:T^*T^*M \longrightarrow T^*TM,
\]
which establishes a natural correspondence between the symplectic manifolds $(T^*T^*M,\omega_{T^*M})$ and $(T^*TM,\omega_{TM})$.  
Locally, this map is given by
\[
{\mathcal I}_{TM}(q,p,\mu_q,\mu_p) = (q,\mu_p,-\mu_q,p),
\]
and satisfies ${\mathcal I}_{TM}^*\omega_{TM} = -\omega_{T^*M}$, hence the name “anti-symplectomorphism”.  
This identification allows us to transfer the symplectic structure between tangent and cotangent bundles, which is essential for defining discrete flows that are consistent with the underlying continuous geometry.

\medskip
The overall structure of the discrete construction is summarized in the commutative diagram of Figure~\ref{Diag:RdT*2}.  
Each horizontal layer corresponds to a different geometric level:  
the bottom layer represents the discrete flow on the configuration manifold $M$,  
the middle layer its tangent lift,  
and the top layer the corresponding lift to the cotangent bundle,  
where the discrete Hamiltonian evolution is defined as a Lagrangian submanifold of $T^*T^*M$.

\begin{figure}[htb!]
\begin{equation*}
\xymatrix{ 
&& {TT^*M } \ar[rr]^{R_d^{T^*}} \ar[d]^{\alpha_{M}} && {T^*M\times T^*M } \ar[d]^{\Phi}
\\ 
T^*T^*M \ar[rr]^{ {\mathcal I}_{TM}} && T^*TM \ar[d]^{\pi_{TM}} \ar[rr]^{ \widehat{R_d}} && T^*(M\times M) \ar[d]^{\pi_{M\times M}}
\\ 
&& TM \ar[rr]^{R_d} && M\times M 
}
\end{equation*}
\caption{Discrete geometric construction based on the discretization map $R_d$.}
\label{Diag:RdT*2}
\end{figure}

\medskip
For non-autonomous systems, we take $M = Q \times \mathbb{R}$ and construct the discretization map $R^M_d$ as the product of two simpler maps:
\[
R^M_d = (R_d^{\mathbb R}, R_d^Q),
\qquad
R_d^{\mathbb R} : T^*\mathbb R \to \mathbb R\times\mathbb R, 
\quad
R_d^Q : T^*Q \to Q\times Q.
\]
This decomposition allows us to encode the evolution of both the configuration variables and the time variable within a unified discrete geometric framework.  
This composition, together with the symplectic identifications described above, leads to a Lagrangian submanifold in $T^*T^*(Q \times \mathbb{R})$ that implicitly defines the discrete evolution map on the extended phase space.  
The next result shows that the discrete flow constructed through this geometric framework defines a canonical transformation on $T^*(Q \times \mathbb{R})$ whose projection onto $T^*Q \times \mathbb{R}$ preserves the geometric properties (volume preservation, Poisson bracket...) associated to the continuous cosymplectic flow.

\begin{thm}\label{mainth}
Given the extended Hamiltonian $\tilde{H}=H\circ \mu+p_t$ then the Lagrangian submanifold 
$
(\Phi^{-1} \circ \widehat{R}^M_d\circ {\mathcal I}_{T(Q \times \mathbb{R})})(h\hbox{d}\tilde{H}(T^*(Q \times \mathbb{R})))
$
implicitly defines a map $\Psi_h:T^*(Q \times \mathbb{R})\rightarrow T^*(Q \times \mathbb{R}) $ which verifies that 
$\Psi_h$ is $\mu$-projectable to a map $\phi_h:T^*Q \times \mathbb{R}\rightarrow T^*Q \times \mathbb{R}$. 
Therefore, $(\Psi_h, \phi_h)$ is a canonical transformation.
\end{thm}

\textit{Proof:} Let $M = Q \times \mathbb R$ and consider the extended Hamiltonian 
\[
\tilde H = H\circ\mu + p_t,
\qquad
\mu(q,t,p_q,p_t) = (q, p_q, t).
\]
We show that the submanifold
\[
\Lambda_h
=
(\Phi^{-1}\circ \widehat R_d^M \circ \mathcal I_{TM})(h\,d\tilde H)
\subset T^*M \times T^*M
\]
is Lagrangian, and that it implicitly defines a symplectic map 
\(\Psi_h:T^*M\to T^*M\) that is $\mu$–projectable.

In local coordinates $(q,t,p_q,p_t)$ on $T^*M$, the differential of $\tilde H$ is
\[
d\tilde H
=
\left(
\frac{\partial H}{\partial q},
\frac{\partial H}{\partial t},
\frac{\partial H}{\partial p_q},
1
\right).
\]
Multiplying by $h$ yields the discrete analogue of the Hamiltonian 1–form:
\[
h\,d\tilde H
=
\bigl(
h\partial_q H,\;
h\partial_t H,\;
h\partial_{p_q}H,\;
h
\bigr).
\]
This is the cotangent vector that encodes a single time step of the discretized dynamics.

The canonical antisymplectomorphism $\mathcal I_{TM}:T^*T^*M \longrightarrow T^*TM$ is locally given by
\[
\mathcal I_{TM}(x,p;\mu,\tilde{\mu})
=
(x,\tilde{\mu};-\mu,p)
\]
and satisfies $\mathcal I_{TM}^*\omega_{TM}=-\omega_{T^*M}$.

Applying $\mathcal I_{TM}$ to $h\,d\tilde H$ produces an element in $T^*TM$:
\[
\mathcal I_{TM}(h\,d\tilde H)
=
\bigl(
q,\;t,\; h\partial_{p_q}H,\; h;\;
-h\partial_q H,\; -h\partial_t H,\;
p_q,\; p_t
\bigr).
\]

Let $R_d = R_d^{\mathbb R}\times R_d^Q : TM \to M\times M$. Its cotangent lift decomposes as $\widehat R_d^M = \widehat R_d^{\mathbb R} \times \widehat R_d^Q$. In local coordinates $(q,t,\dot q,\dot t)$,
\[
R_d(q,t,\dot q,\dot t)
=
\bigl(
(R_d^{Q})^1(q,\dot q),\; (R_d^{\mathbb R})^1(t,\dot t),\;
(R_d^{Q})^2(q,\dot q),\; (R_d^{\mathbb R})^2(t,\dot t)
\bigr),
\]
and its cotangent lift transports momenta via \((T R_d)^{-1}\).

Applying $\widehat R_d^M$ to $\mathcal I_{TM}(h\,d\tilde H)$ gives a point in 
\( T^*(M\times M) \). The map
$\Phi:T^*M\times T^*M\longrightarrow T^*(M\times M)$
is a symplectomorphism whose inverse reorganizes cotangent coordinates to produce pairs $(x,p_x;\; x',p_{x'})$.

Thus
\[
\Lambda_h
=
(\Phi^{-1}\circ\widehat R_d^M\circ\mathcal I_{TM})
(h\,d\tilde H)
\subset T^*M\times T^*M.
\]

Since the 1-form $h\,d\tilde H$ is exact and both $\mathcal I_{TM}$ and $\Phi$ are symplectic or antisymplectic, the image $\Lambda_h$ of $(\Phi^{-1}\circ\widehat R_d^M\circ\mathcal I_{TM})$ is by construction a Lagrangian submanifold of the product symplectic manifold $T^*M\times T^*M$.
It therefore encodes a symplectic map $\Psi_h:T^*M\to T^*M$.

We must now show that:
\[
\mu\circ\Psi_h = \phi_h\circ\mu
\quad\text{for some map }\phi_h:T^*Q \times \mathbb{R}\to T^*Q \times \mathbb{R}.
\]

To show that $\Psi_h$ is $\mu$–projectable, we first investigate the explicit coordinate expression of the Lagrangian submanifold $\Lambda_h$. Using the coordinate formulas for $\widehat R_d^M$ and $\Phi^{-1}$ obtained earlier, we have
\[
\begin{split}
(\Phi^{-1}\circ \widehat{R^M_d})(t, q, \dot t, \dot q; \, & p_t, p_q, p_{\dot t}, p_{\dot q})\\
&=
\begin{pmatrix}
 (R^{\mathbb R}_d)^1(t, \dot t),\;
 -(\mathrm{pr}^{T^*\mathbb R}_1\circ T^*_{(t, \dot t)}{(R^{\mathbb R}_d)^{-1}})(p_t, p_{\dot t})\\[0.25em]
 (R^{Q}_d)^1(q, \dot q),\;
 -(\mathrm{pr}^{T^*Q}_1\circ T^*_{(q, \dot q)}{(R^{Q}_d)^{-1}})(p_q, p_{\dot q})\\[0.25em]
 (R^{\mathbb R}_d)^2(t, \dot t),\;
 (\mathrm{pr}^{T^*\mathbb R}_2\circ T^*_{(t, \dot t)}{(R^{\mathbb R}_d)^{-1}})(p_t, p_{\dot t})\\[0.25em]
 (R^{Q}_d)^2(q, \dot q),\;
 (\mathrm{pr}^{T^*Q}_2\circ T^*_{(q, \dot q)}{(R^{Q}_d)^{-1}})(p_q, p_{\dot q})
\end{pmatrix}.
\end{split}
\]
\color{black}

Substituting the value of ${\mathcal I}_{T(Q \times \mathbb{R})}(h\,d\tilde H)$, namely
\[
(t,q,\; \dot t,\dot q;\; p_t,p_q,p_{\dot t},p_{\dot q})
=
\bigl(
t,q,\;
h,\; h\partial_{p_q}H;\;
-h\partial_tH,\; -h\partial_qH,\;
p_t,\; p_q
\bigr),
\]
\color{black}
we obtain the explicit coordinate expression of $\Lambda_h$:
\[
\Lambda_h=
\begin{pmatrix}
 (R^{\mathbb R}_d)^1(t,h),\;
 -(\mathrm{pr}^{T^*\mathbb R}_1\circ T^*_{(t,h)}(R^{\mathbb R}_d)^{-1})
 \bigl(-h\partial_t H,\; p_t \bigr)
 \\[0.4em]
 (R^{Q}_d)^1(q,h\partial_{p_q}H),\;
 -(\mathrm{pr}^{T^*Q}_1\circ 
 T^*_{(q,h\partial_{p_q}H)}(R^{Q}_d)^{-1})
 \bigl(-h\partial_q H,\; p_q \bigr)
 \\[0.4em]
 (R^{\mathbb R}_d)^2(t,h),\;
 (\mathrm{pr}^{T^*\mathbb R}_2\circ 
 T^*_{(t,h)}(R^{\mathbb R}_d)^{-1})
 \bigl(-h\partial_t H,\; p_t \bigr)
 \\[0.4em]
 (R^{Q}_d)^2(q,h\partial_{p_q}H),\;
 (\mathrm{pr}^{T^*Q}_2\circ 
 T^*_{(q,h\partial_{p_q}H)}(R^{Q}_d)^{-1})
 \bigl(-h\partial_q H,\; p_q \bigr)
\end{pmatrix}.
\]
\color{black}

From these expressions we see that the update of the configuration variables \(q\) and the momentum \(p_q\)  
depends on \(q,p_q,t\) and the Hamiltonian derivatives,  
\emph{but no term involving \(p_t\) appears in the \(Q\)-component updates}.  
In particular,
\[
q_{k+1}
= (R_d^Q)^2\bigl(q,\; h\partial_{p_q}H\bigr),
\qquad
p_{q,k+1}
= (\mathrm{pr}^{T^*Q}_2\circ T^*R_d^Q)(-h\partial_q H,\; p_q),
\]
contains no $p_t$. In addition, the variable \(p_t\) appears \emph{only} inside the cotangent components of the 
\(\mathbb R\)-part, i.e., only in the expressions updating \(p_{t,k}\) and \(p_{t,k+1}\).  
Thus, the dependence of the discrete dynamics on \(p_t\) is vertical and does not couple into the $Q$–sector.

Hence, the full discrete map necessarily factorizes as
\[
\Psi_h(q,t,p_q,p_t)
=
\bigl(
\phi_h(q,t,p_q),\;
p_t + F(q,t,h,p_q)
\bigr),
\]
for some smooth function $F$ that does not depend on $p_t$.

Applying $\mu$ to both sides,
\[
\mu\circ\Psi_h(q,t,p_q,p_t)
= 
\phi_h(q,t,p_q).
\]
Thus $\Psi_h$ is $\mu$-projectable.\hfill$\square$


As a concrete example, consider $M=\mathbb R^n$ and apply the midpoint discretization map
\[
R_d^M(q,t,\dot q,\dot t)
=
\bigl(q-\tfrac{\dot q}{2}\,;\;t-\tfrac{\dot t}{2},\;
      q+\tfrac{\dot q}{2}\,;\; t+\tfrac{\dot t}{2} \bigr),
\]
which corresponds to evaluating the discrete step at the midpoint of the segment determined by $(q,t)$ and the tangent vector $(\dot q,\dot t)$.  
Substituting into the general construction of Theorem~\ref{mainth}, the Lagrangian submanifold 
$\Lambda_h$ yields the following implicit discrete update for the extended Hamiltonian system.

Indeed, since
\[
\dot t = h,\qquad
\dot q = h\,\frac{\partial H}{\partial p}(q,t,p),
\]
and the cotangent components are transformed by the pullback of $(R_d^M)^{-1}$, 
the coordinate expression of $\Lambda_h$ leads to the midpoint evaluation
\[
(q_{k+\frac12}, t_{k+\frac12}, p_{k+\frac12})
=
\Bigl(
\tfrac{q_k+q_{k+1}}{2},\;
\tfrac{t_k+t_{k+1}}{2},\;
\tfrac{p_k+p_{k+1}}{2}
\Bigr).
\]
Hence the discrete Hamilton equations produced by the geometric construction read:
\begin{align*}
\frac{q_{k+1}-q_k}{t_{k+1}-t_k}
&=
\frac{\partial H}{\partial p}\bigl(
q_{k+\frac12},\; t_{k+\frac12},\; p_{k+\frac12}
\bigr),\\[0.5em]
\frac{p_{k+1}-p_k}{t_{k+1}-t_k}
&=
-\frac{\partial H}{\partial q}\bigl(
q_{k+\frac12},\; t_{k+\frac12},\; p_{k+\frac12}
\bigr),
\end{align*}
which is precisely the \textit{implicit midpoint rule} for the non-autonomous Hamiltonian system.

Because this method arises from a Lagrangian submanifold defined through a discrete canonical transformation on the extended phase space $T^*(Q \times \mathbb{R})$, it automatically inherits the geometric preservation properties established in the next corollary:  
volume preservation on $T^*Q \times \mathbb{R}$, preservation of the Poisson bracket $\{\cdot,\cdot\}_0$, and symplecticity on each time slice $T^*Q \times \{t\}$. This is reflected with the following corollary.

\begin{cor}
The numerical method $\phi_h$ satisfies the following properties: 
\begin{enumerate}
\item it preserves the volume form $\mathrm{Vol}_0=\omega_0^{\,n}\wedge\eta$;
\item it preserves the Poisson bracket $\{\cdot,\cdot\}_0$ on $T^*Q \times \mathbb{R}$;
\item for each fixed $t\in\mathbb R$, the induced map $\phi_{t,h}:T^*Q\to T^*Q$ is a symplectomorphism. 
\end{enumerate}
\end{cor}
\begin{proof}

(1) and (2) are  direct consequences of Propositions \ref{Prop5} and \ref{Propo4}, respectively. 

\medskip

For the proof of (3),  we have that for each fixed $t\in\mathbb R$
\[
\phi_{t,h}=\mathrm{pr}_{T^*Q}\circ\phi_h\circ\iota_t,
\]
where $\iota_t:T^*Q\to T^*Q \times \mathbb{R}$ is the inclusion $\iota_t(\alpha)=(t;\alpha)$.  
Let $F,G\in C^\infty(T^*Q)$, and consider their time–independent lifts 
$f=\mathrm{pr}_{T^*Q}^*F$ and $g=\mathrm{pr}_{T^*Q}^*G$ on $T^*Q \times \mathbb{R}$.
For these functions the cosymplectic bracket reduces to the canonical one:
\[
\{f,g\}_0=\mathrm{pr}_{T^*Q}^*\{F,G\}_{\omega_Q}.
\]
From part (2) we know $\phi_h$ is a Poisson map for $\{\cdot,\cdot\}_0$; pulling back by $\iota_t$ gives
\[
\begin{aligned}
    \iota_t^*\{f\circ\phi_h,g\circ\phi_h\}_0
&=\iota_t^*\big(\mathrm{pr}_{T^*Q}^*\{F,G\}_{\omega_Q}\circ\phi_h\big)
=\{F,G\}_{\omega_Q}\circ(\mathrm{pr}_{T^*Q}\circ\phi_h\circ\iota_t)\\
&=\{F,G\}_{\omega_Q}\circ\phi_{t,h}.
\end{aligned}
\]
On the other hand,
\[
\iota_t^*\{f\circ\phi_h,g\circ\phi_h\}_0
=\{\iota_t^*(f\circ\phi_h),\iota_t^*(g\circ\phi_h)\}_{\omega_Q}
=\{F\circ\phi_{t,h},G\circ\phi_{t,h}\}_{\omega_Q},
\]
because $\iota_t^*\omega_0=\omega_Q$ and $\iota_t$ intertwines the brackets for time–independent functions.

Combining the previous identities,
\[
\{F\circ\phi_{t,h},G\circ\phi_{t,h}\}_{\omega_Q}
=\{F,G\}_{\omega_Q}\circ\phi_{t,h},
\qquad\forall F,G\in C^\infty(T^*Q).
\]
Therefore $\phi_{t,h}$ is a Poisson diffeomorphism of $(T^*Q,\omega_Q)$.  
Since the Poisson bracket defined by $\omega_Q$ is non-degenerate, this is equivalent to
\[
\phi_{t,h}^*\omega_Q=\omega_Q,
\]
so $\phi_{t,h}$ is a symplectomorphism. 

\end{proof}

\section{Numerical Simulations}\label{section: Numerical Simulations}

In this section we present two numerical examples to illustrate the performance and qualitative behavior of the proposed geometric integrators. These examples are chosen because they combine simplicity of interpretation with nontrivial dynamical features arising from dissipative effects and time-dependent potentials. Moreover, these are prototypical models in mechanical systems where energy dissipation and external potential play a central role, and the preservation of geometric structures (such as volume or Poisson structure) can significantly influence the accuracy and long-term stability of numerical simulations.
The simulations are carried out using an implicit symplectic Euler-type geometric integrator, which can be schematized as follows
\begin{equation}\label{eq: euler-type}
\begin{split}
    q(i+1) &= q(i) + h \tilde{H}_p(q(i), t(i), p(i+1), p_{t}(i+1) ),\\
    t(i+1) &= t(i) + h \tilde{H}_{p_t}(q(i), t(i), p(i+1), p_{t}(i+1) ),\\
    p(i+1) &= p(i) - h \tilde{H}_q (q(i), t(i), p(i+1), p_{t}(i+1) ),\\
    p_{t}(i+1) &= p_{t}(i) - h \tilde{H}_{t}(q(i), t(i), p(i+1), p_{t}(i+1) ).
\end{split}
\end{equation}
for an arbitrary Hamiltonian $\tilde{H}: T^*(Q\times \mathbb{R})\rightarrow \mathbb{R}$.
These symplectic integrators are compared against a standard fourth-order Runge–Kutta (RK4) scheme. The RK4 method provides a reference in terms of local accuracy, but it does not preserve the underlying geometry. The comparison, therefore, highlights the qualitative differences in the long-term behavior.

\subsection{Damped harmonic oscillator}\label{subsec: Damped harmonic oscillator}

We first consider a damped harmonic oscillator, a classical mechanical system used to describe oscillations under the influence of linear friction. This case provides a simple yet meaningful benchmark to test the consistency of the proposed discretization in the presence of an explicit time dependence in the Lagrangian. The formulation follows the approach in \cite{kanai1948quantization}, where the damping is incorporated through a time-exponential factor that modifies both kinetic and potential terms.
The Lagrangian of the system is given by
\begin{equation*}
    \begin{split}
        L(t,q,\dot{q}) = e^{-\gamma t} \left( \frac{1}{2} m\dot{q}^2 - \frac{1}{2} m \omega_0 q^2 \right)
    \end{split}
\end{equation*}
where $m$ is the mass, $\omega_0$ the undamped natural frequency, and $\gamma$ the damping constant. The exponential factor $e^{-\gamma t}$ models the continuous decay of energy due to frictional losses, making the system explicitly time-dependent and non-conservative.
The corresponding extended Hamiltonian is
\begin{equation*}
    \begin{split}
        H(q, t, p, p_t) = \frac{1}{2}e^{\gamma t} \frac{p^2}{m} + \frac{1}{2}e^{-\gamma t} m \omega_0 q^2 + p_t.
    \end{split}
\end{equation*}
The dynamic equations for the system are
\begin{equation*}
    \begin{split}
        \dot{q} = \frac{1}{m} e^{\gamma t} p, 
        \quad
        \dot{t} = 1, \quad
        \dot{p} = -m \omega_0 e^{-\gamma t} q, \quad
        \dot{p}_t = -\frac{\gamma}{2m} e^{\gamma t} p^2 + \frac{\gamma m \omega_0}{2} e^{-\gamma t} q^2.
    \end{split}
\end{equation*}
For these simulations, we set the parameters
$m = 1 \,\mathrm{kg}$,
$\omega_0 = 1 \,\mathrm{rad/s}$,
$\gamma = 0.2$;
and the initial conditions
$q(0) = 1 \,\mathrm{rad}$,
$p(0) = 1 \,\mathrm{kg \, m^2 \, rad/s}$.
Both integrators use a time step of $T_s = 0.1 \,\mathrm{s}$ over a total integration horizon of $T_{\text{fin}} = 50  \,\mathrm{s}$.
\begin{figure}[t!]
\centering
\begin{subfigure}{\textwidth}
    \includegraphics[width=0.49\linewidth]{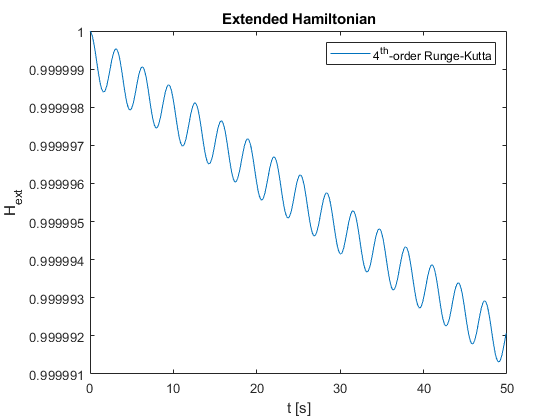}
    \includegraphics[width=0.49\linewidth]{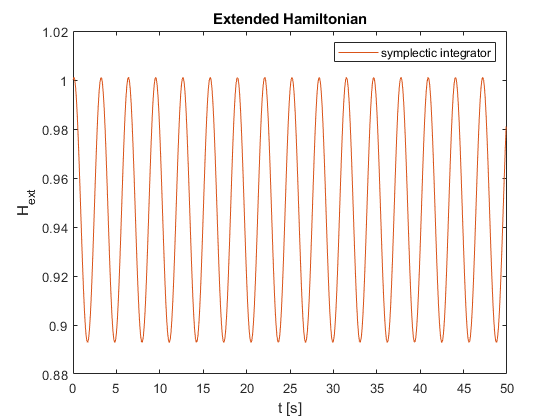}
    \caption{Extended Hamiltonian}
    \label{fig:DampedExtHam}
\end{subfigure}
\begin{subfigure}{0.49\textwidth}
    \includegraphics[width=\linewidth]{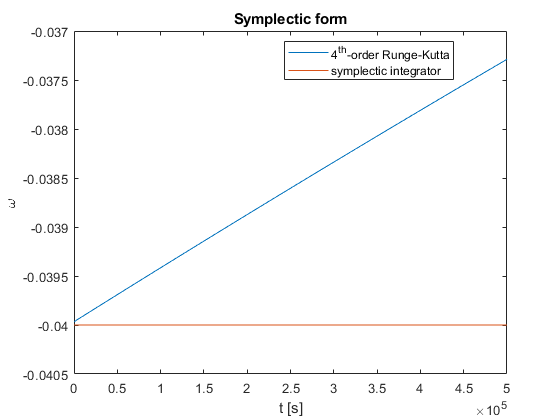} 
    \caption{Symplectic form}
    \label{fig:SymFormDamped}
\end{subfigure}%
\hfill
\begin{subfigure}{0.49\textwidth}
    \includegraphics[width=\linewidth]{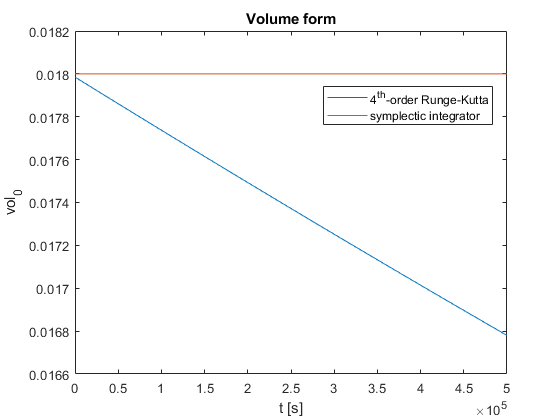}
    \caption{Volume form}
    \label{fig:VolumeFormDamped}
\end{subfigure}
\caption{Extended Hamiltonian, symplectic form, and volume form comparison between 4th order Runge-Kutta and the symplectic integrators in the damped harmonic oscillator case.}
\label{fig:Damped}
\end{figure}
The time evolutions of the extended Hamiltonians are shown in Figure \ref{fig:DampedExtHam}.
In Figure \ref{fig:SymFormDamped} we show the evolution of the symplectic form $\omega$ computed for the RK4 and the symplectic integrator.
Two motions with close initial conditions are considered.
The associated vector fields were approximated through finite differences between consecutive states of each trajectory
\begin{equation*}
    X_{k+1}^i = \begin{bmatrix}
        q_{k+1}\\
        p_{k+1}\\
    \end{bmatrix}-\begin{bmatrix}
        q_{k}\\
        p_{k}\\
    \end{bmatrix}, \quad i = 1,2.
\end{equation*}
The preservation of the discrete symplectic form was then verified by evaluating these vector fields on the symplectic form $\omega (X_1, X_2)$.
In Figure \ref{fig:VolumeFormDamped} we present the evolution of the volume form $\mathrm{vol}_0$.
The comparison includes the proposed RK4 scheme and the symplectic integrator.
For this analysis, three trajectories were integrated starting from slightly perturbed initial conditions. 
The associated vector fields were approximated through finite differences between consecutive states of each trajectory
\begin{equation*}
    X_{k+1}^i = \begin{bmatrix}
        q_{k+1}\\
        t_{k+1}\\
        p_{k+1}\\
    \end{bmatrix}-\begin{bmatrix}
        q_{k}\\
        t_{k}\\
        p_{k}\\
    \end{bmatrix}, \quad i = 1,2,3.
\end{equation*}
The discrete volume form is then evaluated as the determinant of the matrix composed of these approximate vector fields
\begin{equation*}
    \mathrm{vol}_0 = \det \begin{bmatrix}
        X^1 \, X^2 \, X^3
    \end{bmatrix}.
\end{equation*}
For these simulations, we change the damping constant and final time horizon, choosing 
$\gamma = 2 \cdot 10^{-4}$
and
$T_{\text{fin}} = 5 \cdot 10^5 \,\mathrm{s}$.

\subsection{Harmonic oscillator with time-dependent potential}\label{subsec: Harmonic oscillator with time-dependent potential}

We consider a harmonic oscillator with a time-dependent potential energy as in \cite{marthinsen2016geometric}.
In this example the time dependency affects only the potential energy of the system and does not impact the conjugate momenta.
The Hamiltonian of the system has the form
\[
H(q, t, p) = \frac{1}{2} \left( (1 + \varepsilon \sin{(\alpha t)}) q^2 + p^2 \right)
\]
where $\alpha, \varepsilon \in \mathbb{R}$ are constant parameters.
The extended Hamiltonian becomes
\[
H(q, t, p, p_t) = \frac{1}{2} \left( (1 + \varepsilon \sin{(\alpha t)}) q^2 + p^2 \right) + p_t
\]
with dynamic equations
\[
\dot{q} = p,\quad 
\dot{t} = 1,\quad 
\dot{p} = -(1 + \varepsilon \sin{(\alpha t)}) q,\quad
\dot{p}_t = -\frac{1}{2} \alpha\varepsilon \cos{(\alpha t)} q^2.
\]
We set as parameters
$\varepsilon = 0.0035$,
$\alpha = 0.0008$,
and as initial conditions
$q(0) = 1 \,\mathrm{rad}$,
$p(0) = 1 \,\mathrm{kg \, m^2 \, rad/s}$.
For the time parameters, we choose
$T_s = 0.1 \, \mathrm{s}$, and
$T_{fin} = 50 \, \mathrm{s}$.
\begin{figure}[t!]
\centering
\begin{subfigure}{\textwidth}
    \includegraphics[width=0.49\linewidth]{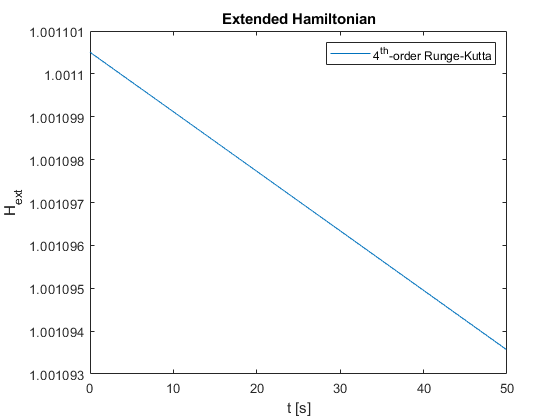} 
    \includegraphics[width=0.49\linewidth]{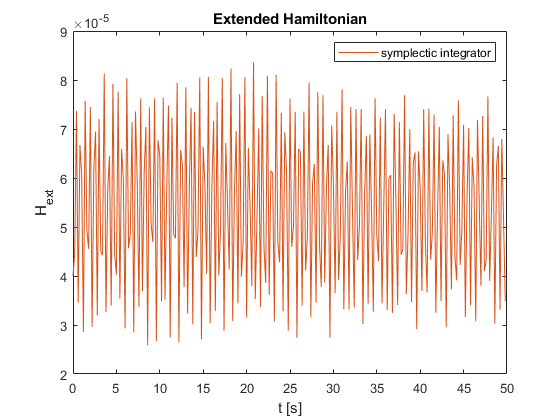}
    \caption{Extended Hamiltonian}
    \label{fig:ExtHamTimeDepHarmOsc}
\end{subfigure}
\begin{subfigure}{0.49\textwidth}
    \includegraphics[width=\linewidth]{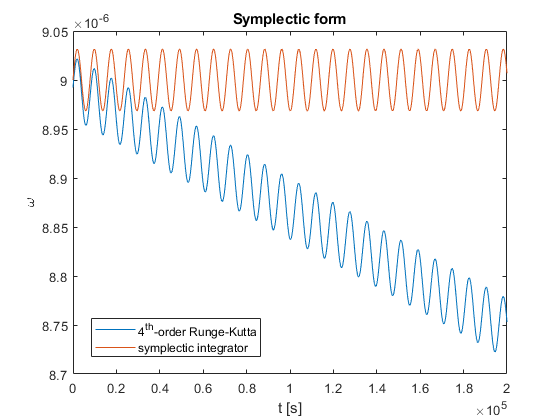} 
    \caption{Symplectic form}
    \label{fig:SymFormTimeDepHarmOsc}
\end{subfigure}%
\hfill
\begin{subfigure}{0.49\textwidth}
    \includegraphics[width=\linewidth]{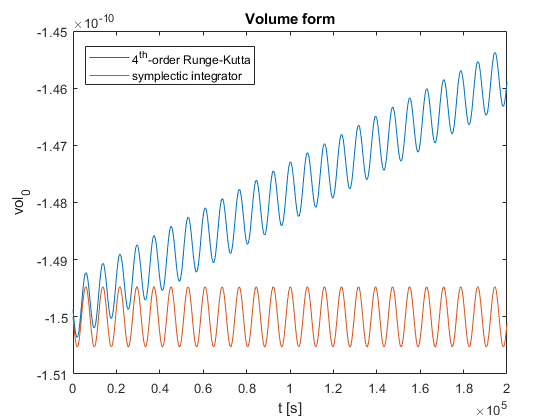}
    \caption{Volume form}
    \label{fig:VolumeFormTimeDepHarmOsc}
\end{subfigure}
\caption{Extended Hamiltonian, symplectic form, and volume form comparison between 4th order Runge-Kutta and the symplectic integrators in the harmonic oscillator with time-dependent potential case.}
\label{fig:TimeDepHarmOsc}
\end{figure}
Figure \ref{fig:ExtHamTimeDepHarmOsc} shows a comparison between the evolutions of the extended Hamiltonians computed using RK4 and the symplectic integrator.
In Figure \ref{fig:SymFormTimeDepHarmOsc} and Figure \ref{fig:VolumeFormTimeDepHarmOsc}, we show the time evolution of the symplectic and volume forms computed as in the previous example.
For these simulations, we consider a longer time horizon of $T_{fin} = 2\cdot 10^5 \, \mathrm{s}$.

\medskip

In both the examples \ref{subsec: Damped harmonic oscillator} and \ref{subsec: Harmonic oscillator with time-dependent potential} the numerical simulations produced the expected results.
The symplectic numerical integrator was designed to preserve the geometric structure, and this is verified by observing the behaviours of the symplectic and volume forms.
The latter, indeed, present a dynamic that on average tend to maintain a constant value, with small oscillations due to numerical approximations.
The Runge-Kutta method, instead, despite its high order of integration, shows dynamics that tend to diverge.
Concerning the extended Hamiltonian, it is sufficient to consider a shorter time horizon to observe that, even in this case, the Runge-Kutta method produces a slight bias that can lead to incorrect values.
It is known, instead, that the symplectic method does not preserve the Hamiltonian value exactly, but produces oscillations around an average value that coincides with the correct one.




\section{Application to an electric particle in an electromagnetic field}\label{section: Application to an electric particle in an electromagnetic field}

In this section, we show how the theory developed in this work aligns with gauge-invariant systems.
We present a particular application of the structure-preserving integrators to the electromagnetic field.
In particular, we regard the case of different gauge choices for the system, and we evaluate which quantities are preserved under transformations from one gauge to another.

We consider a charged particle moving in space $Q := \mathbb{R}^3$ under the action of a time-dependent electromagnetic field described in terms of a scalar $\phi(q,t)$ and a vector potential $A(q,t) = (A_1(q,t),A_2(q,t),A_3(q,t))$\footnote{
We recall that the electric and magnetic fields are given in terms of the scalar and vector potentials by
\[
E(q,t) = -\nabla \phi(q,t)-\partial_t A(q,t),
\qquad
B(q,t) = \nabla \times A(q,t).
\]
}, where $q=(x,y,z)$ represents the three-dimensional coordinates.
The Lagrangian of the system is given by
\begin{align*}
L(q,\dot q,t)
&=\frac12 m\norm{\dot q}^2+eA(q,t)\cdot\dot q-e\phi(q,t).
\end{align*}
We stress that the potentials are not uniquely determined, but can be chosen and modified according to the so-called gauge transformations.
Let $\chi=\chi(q,t)$ be a smooth function on $Q \times \mathbb{R}$. A gauge transformation defined by
\begin{equation*}\label{eq:gauge-transform-potentials}
A'(q,t)=A(q,t)+\nabla\chi(q,t),
\qquad
\phi'(q,t)=\phi(q,t)-\partial_t\chi(q,t)
\end{equation*}
leaves the electromagnetic fields unchanged
and correspondingly the Lagrangian changes by a total derivative (see, e.g. \cite{Goldstein2001}),
therefore the Euler-Lagrange equations remain unchanged.

\subsection{Gauge transformations as canonical transformations}

The key observation for the extended Hamiltonian formalism is that a gauge transformation induces a canonical transformation on the extended phase space.
Let
\begin{equation*}\label{eq:extended-hamiltonian-general}
\wtH(q,p,t,p_t)=\frac{1}{2m}\norm{p-eA(q,t)}^2+e\phi(q,t)+p_t
\end{equation*}
denote the extended Hamiltonian, where $p = (p_x, p_y, p_z) = \nabla_{\dot{q}} L$
indicate the canonical momenta. 
The canonical symplectic form on the extended phase space is
\begin{equation}\label{eq: canonical form electric}
    \omega = \mathrm{d} x \wedge \mathrm{d} p_x + \mathrm{d} y \wedge \mathrm{d} p_y + \mathrm{d} z \wedge \mathrm{d} p_z + \mathrm{d} t \wedge \mathrm{d} p_t.
\end{equation}
Given a gauge function $\chi(q,t)$, we define a gauge-dependent function on the phase state
\begin{equation}\label{eq:gauge-canonical-map}
G_\chi(q,p,t,p_t)=\bigl(q,\,p+e\nabla\chi(q,t),\,t,\,p_t+e\partial_t\chi(q,t)\bigr).
\end{equation}
The following proposition explains the relation between gauge and canonical transformations.
\begin{prop}\label{prop:gauge-canonical}
Let $(\phi,A)$ and $(\phi',A')$ be two pairs of scalar and vector potentials related by the gauge transformation induced by a function $\chi$ on $Q \times \mathbb{R}$
\[
A'=A+\nabla\chi,
\qquad
\phi'=\phi-\partial_t\chi.
\]
Then the map $G_\chi$ defined in \eqref{eq:gauge-canonical-map} is a canonical transformation on the extended phase space, and the corresponding extended Hamiltonians satisfy
\begin{equation}\label{eq:hamiltonians-conjugate}
\wtH'\circ G_\chi=\wtH.
\end{equation}
\end{prop}

\begin{proof}
We first check the Hamiltonian identity. By definition,
\[
\wtH'(q,p',t,p_t')
=
\frac{1}{2m}\norm{p'-eA'(q,t)}^2 + e\phi'(q,t) + p_t'.
\]
Substituting
\[
p'=p+e\nabla\chi(q,t),
\quad
p_t'=p_t+e\partial_t\chi(q,t),
\quad
A'=A+\nabla\chi,
\quad
\phi'=\phi-\partial_t\chi,
\]
we obtain
\begin{align*}
\wtH'\circ G_\chi
&=
\frac{1}{2m}\norm{p+e\nabla\chi-e(A+\nabla\chi)}^2
+e(\phi-\partial_t\chi)+p_t+e\partial_t\chi \\
&=
\frac{1}{2m}\norm{p-eA}^2+e\phi+p_t = \wtH.
\end{align*}
This proves \eqref{eq:hamiltonians-conjugate}.

To verify that $G_\chi$ is canonical, consider the canonical one-form on the extended phase space,
\[
\Theta=p\cdot \dd q+p_t\,\dd t.
\]
Then
\begin{align*}
G_\chi^*\Theta
&=(p+e\nabla\chi)\cdot \dd q+(p_t+e\partial_t\chi)\,\dd t
=p\cdot\dd q+p_t\,\dd t+e\bigl(\nabla\chi\cdot\dd q+\partial_t\chi\,\dd t\bigr)\\
&=\Theta+e\,\dd\chi.
\end{align*}
Taking exterior derivative gives
\[
G_\chi^*(\omega) =
G_\chi^*(-\dd\Theta) =
-\dd\Theta
= \omega,
\]
so $G_\chi$ preserves the canonical symplectic form \eqref{eq: canonical form electric}.
Hence $G_\chi$ is canonical.
\end{proof}

\begin{remark}\label{rem:gauge-physical-vs-canonical}
We highlight here the quantities that are preserved through the map $G_\chi$.
Observe that the canonical momentum is not gauge invariant. Indeed, under the gauge
transformation
\[
A' = A+\nabla\chi, \qquad \phi'=\phi-\partial_t\chi,
\]
the corresponding canonical transformation on the extended phase space is
\[
G_\chi(q,p,t,p_t)
=
\bigl(q,\; p+e\nabla\chi(q,t),\; t,\; p_t+e\partial_t\chi(q,t)\bigr).
\]
Therefore, if \(p'=p+e\nabla\chi(q,t)\), then the mechanical momentum satisfies
\[
\Pi'
=
p'-eA'(q,t)
=
p+e\nabla\chi(q,t)-e\bigl(A(q,t)+\nabla\chi(q,t)\bigr)
=
p-eA(q,t)
=
\Pi .
\]
Thus, the gauge-invariant momentum is the mechanical momentum
\[
\Pi = p-eA(q,t),
\]
whereas the canonical momentum \(p\) itself depends on the chosen gauge.
Consequently, when comparing numerical solutions obtained in different gauges,
one should compare the physical observables \(q(t)\) and \(\Pi(t)\), or else first
transform the canonical momenta according to the gauge map \(G_\chi\).
\end{remark}


Proposition \ref{prop:gauge-canonical} states that it's possible to switch from a gauge choice to another while preserving the canonical structure.
We state here the discrete counterpart.

\begin{prop}\label{prop:transported-discrete-map}
If $\Phi_h^{A}$ is a canonical map on the extended phase space, then the transported map $\widehat{\Phi}_h^{B}$ defined by 
\begin{equation}\label{eq:discrete-gauge-transport}
\widehat{\Phi}_h^{B}
=
G_{\chi,t_{n+1}}\circ \Phi_h^{A} \circ G_{\chi,t_n}^{-1},
\end{equation}
where $G_{\chi,t_n}$ denotes the canonical gauge transformation evaluated at time $t_n$, is also canonical. Moreover, by construction, the discrete evolutions in the two gauges $z_n^{A}$ and $z_n^{B}$ are exactly related by the gauge transformation at the discrete times.
\end{prop}

\begin{proof}
Since each $G_{\chi,t_n}$ is canonical and $\Phi_h^{A}$ is canonical, the composition
\[
G_{\chi,t_{n+1}}\circ \Phi_h^{A} \circ G_{\chi,t_n}^{-1}
\]
is canonical as well. The second statement follows directly from the definition: if
\[
z_n^{B} = G_{\chi,t_n}(z_n^{A}),
\]
then
\[
z_{n+1}^{B}
=
\widehat{\Phi}_h^{B}(z_n^{B})
=
G_{\chi,t_{n+1}}\bigl(\Phi_h^{A}(z_n^{A})\bigr)
=
G_{\chi,t_{n+1}}(z_{n+1}^{A}).
\]
Thus the two discrete trajectories remain exactly gauge-related at every step.
\end{proof}
Proposition \ref{prop:transported-discrete-map} clarifies that, rather than discretizing independently in two gauge-equivalent coordinate descriptions, one may transport the discrete canonical map by conjugation.
This operation preserves the volume form, the Poisson bracket and hence the symplectic structure.

\begin{remark}\label{rem:naive-vs-transported}
We underline that the transported discrete map $\widehat{\Phi}_h^{2}$ does not, in general, coincide with the map obtained by applying the same rule directly to $\Phi_h^{2}$.
The discrepancy comes from the explicit time dependence of the gauge transformation, and depends on the order of the geometric integrator.
\end{remark}


\subsection{Gauge analysis}\label{subsec: Gauge analysis}

In this section, we present several simulations that empirically validate the results obtained in the previous section.
We start by considering the following vector and scalar potentials (we will refer to this choice as `original gauge')
\[ 
A(q,t) = \beta(t)\begin{bmatrix}
    x-y \\ y+x \\ 0
\end{bmatrix}, \quad 
\phi(q,t) = 0,
\]
where $\beta(t)$ is a time-dependent function.
With these potentials, the electric and magnetic fields are given by
\[
E(q,t) 
=
-\dot\beta(t)
\begin{bmatrix}
x-y \\ x+y \\ 0    
\end{bmatrix},
\quad
B(q,t)
=
\beta(t)
\begin{bmatrix}
0 \\ 0 \\ 2    
\end{bmatrix},
\]
and thus the model is non-autonomous and electromagnetic.
The Lagrangian of the system reads
\begin{equation*}\label{eq:lagrangian-initial}
L(x,y,\dot x,\dot y,t)
=
\frac12 m(\dot x^2+\dot y^2)
+e\beta(t)((x-y)\dot x + (x+y)\dot y)
\end{equation*}
with canonical momenta
\begin{equation*}\label{eq:canonical-momenta-initial}
\begin{split}
p_x =& \frac{\partial L}{\partial \dot{x}} = m\dot x + e\beta(t)(x-y),
\\
p_y =& \frac{\partial L}{\partial \dot{y}} = m\dot y + e\beta(t)(x+y),
\end{split}
\end{equation*}
and the mechanical momenta
\[
\begin{split}
\Pi_x := p_x- eA_1 =& p_x - e\beta(t)(x-y),
\\
\Pi_y := p_y- eA_2 =& p_y - e\beta(t)(x+y).
\end{split}
\]
The non-autonomous Hamiltonian is then
\begin{equation*}\label{eq:hamiltonian-initial}
H(x,y,p_x,p_y,t)
=
\frac{1}{2m}\left(p_x - e\beta(t)(x-y)\right)^2
+
\frac{1}{2m}\left(p_y - e\beta(t)(x+y)\right)^2,
\end{equation*}
and the associated extended Hamiltonian is
\begin{equation*}\label{eq:extended-hamiltonian-initial}
\wtH(x,y,p_x,p_y,t,p_t) = H(x,y,p_x,p_y,t)+p_t.
\end{equation*}
We consider the following choices of gauges, with their respective vector and scalar potentials (we drop their coordinates and time dependencies).
\begin{enumerate}
    \item Coulomb, characterized by the null gradient of the vector potential, i.e., $\nabla \cdot A_{\text{c}} = 0$:
    \begin{align*}
    \chi_{\text{c}} = -x^2 \beta(t),
    \quad
    A_{\text{c}} = \beta(t)\begin{bmatrix}
    -x-y \\ y+x \\ 0
    \end{bmatrix}, 
    \quad
    \phi_{\text{c}} = x^2 \dot\beta(t);
    \end{align*}
    \item symmetric, that requests $\partial_y (A_{\text{s}})_1 + \partial_x (A_{\text{s}})_2 = 0$:
    \begin{align*}
    \chi_{\text{s}} = -\frac{1}{2}(x^2+y^2) \beta(t),
    \quad
    A_{\text{s}} = \beta(t)\begin{bmatrix}
    -y \\ x \\ 0
    \end{bmatrix},
    \quad
    \phi_{\text{s}} = \frac{1}{2}(x^2+y^2) \dot{\beta}(t);
    \end{align*}
    \item Landau, with vector potential linear in $x$ in the second component and zero in the others, i.e., $(A_{\text{La}})_2 = f(t)x$, $(A_{\text{La}})_1 = (A_{\text{La}})_3 = 0$:
    \begin{align*}
    \chi_{\text{La}} =
    -\frac{1}{2}(x-y)^2\beta(t),
    \quad
    A_{\text{La}} = \beta(t)\begin{bmatrix}
    0 \\ 2x \\ 0
    \end{bmatrix},
    \quad
    \phi_{\text{La}} = \frac{1}{2}(x-y)^2 \dot{\beta}(t);
    \end{align*}
    \item Lorenz, with scalar and vector potential invariant under Lorez transformations, i.e., $\nabla \cdot A_{\text{Lo}} + \partial_t \phi_{\text{Lo}} = 0$:
    \begin{align*}
    \chi_{\text{Lo}} = - 2 \int^t_0 \beta(\tau) d \tau,
    \quad
    A_{\text{Lo}} = \beta(t)\begin{bmatrix}
    x-y \\ y+x \\ 0
    \end{bmatrix},
    \quad
    \phi_{\text{Lo}} = 2 \beta(t).
    \end{align*}
\end{enumerate}
The corresponding non-autonomous extended Hamiltonian are
\[
\begin{split}
\tilde H_{\mathrm{c}}(x,y,p_x,p_y,t,p_t)
=&
\frac{1}{2m}\left(p_x - e\beta(t)(-x-y)\right)^2\\
&+
\frac{1}{2m}\left(p_y - e\beta(t)(x+y)\right)^2
+
ex^2 \dot\beta(t) + p_t,
\\
\tilde H_{\mathrm{s}}(x,y,p_x,p_y,t,p_t)
=&
\frac{1}{2m}\left(p_x + e\beta(t)y\right)^2 +
\frac{1}{2m}\left(p_y - e\beta(t)x\right)^2\\
&+
\frac{e}{2}(x^2+y^2) \dot\beta(t) + p_t,
\\
\tilde H_{\mathrm{La}}(x,y,p_x,p_y,t,p_t)
=&
\frac{1}{2m}\left(p_x\right)^2+
\frac{1}{2m}\left(p_y - 2e\beta(t)x\right)^2\\
&+
\left(xy - \frac{1}{2}(x^2+y^2)\right) \dot{\beta}(t)
+ p_t,\\
\tilde H_{\mathrm{Lo}}(x,y,p_x,p_y,t,p_t)
=&
\frac{1}{2m}\left(p_x - e\beta(t)(x-y)\right)^2\\
&+
\frac{1}{2m}\left(p_y - e\beta(t)(x+y)\right)^2
+
2e \beta(t) + p_t.
\end{split}
\]

\subsection{Numerical experiments}

We compare the results obtained by applying the symplectic integrator to Hamiltonian systems with the different gauge choices.
We consider the vector and scalar potentials defined in Section \ref{subsec: Gauge analysis} with the function $\beta$ defined as
\[
\beta(t) = B_0 (1 + \varepsilon \sin(\omega t)).
\]
We choose the parameters
$B_0 = 0.05$,
$\varepsilon = 10$,
$\omega = 2$, 
the mass and the charge are set to
$m = 1 \mathrm{kg}$,
$e = 1  \mathrm{C}$,
while for the initial conditions, we put
$x(0) = 0.1 \mathrm{m}$,
$y(0) = 0.4 \mathrm{m}$,
$p_x(0) = 0.5 \mathrm{kg \, m/s}$,
and 
$p_y(0) = -0.2 \mathrm{kg \, m/s}$.
We implement the Euler-type integrator as defined in \eqref{eq: euler-type} with an integration step $T_s = 0.001 \mathrm{s}$, and for a time horizon of $T_{fin} = 20 \mathrm{s}$.
The trajectory simulations for the different gauge choices are shown in Figure \ref{fig:gauge1}.
Here, we compare the trajectories for $x$ and $y$, the canonical momenta $p_x$ and $p_y$, and the mechanical momenta $\Pi_x$ and $\Pi_y$.
The experimental results show the behaviors that we expect.
The trajectories on the $xy$ plane are the same for all the choices of the gauge, as well as the mechanical momenta $\Pi_x$ and $\Pi_y$, with a small deviation depending on the integration order.
Figure \ref{fig:gauge2} highlights these differences by showing how the various trajectories, obtained with different gauge choices, differ from the original one, 
as prescribed by the Remark \ref{rem:naive-vs-transported}.
Moreover, we can observe clearly different behaviors regarding canonical momenta as highlighted by Remark \ref{rem:gauge-physical-vs-canonical}.

\begin{figure}[t!]
\centering
\begin{subfigure}{\textwidth}
\centering
    \includegraphics[width=0.85\linewidth]{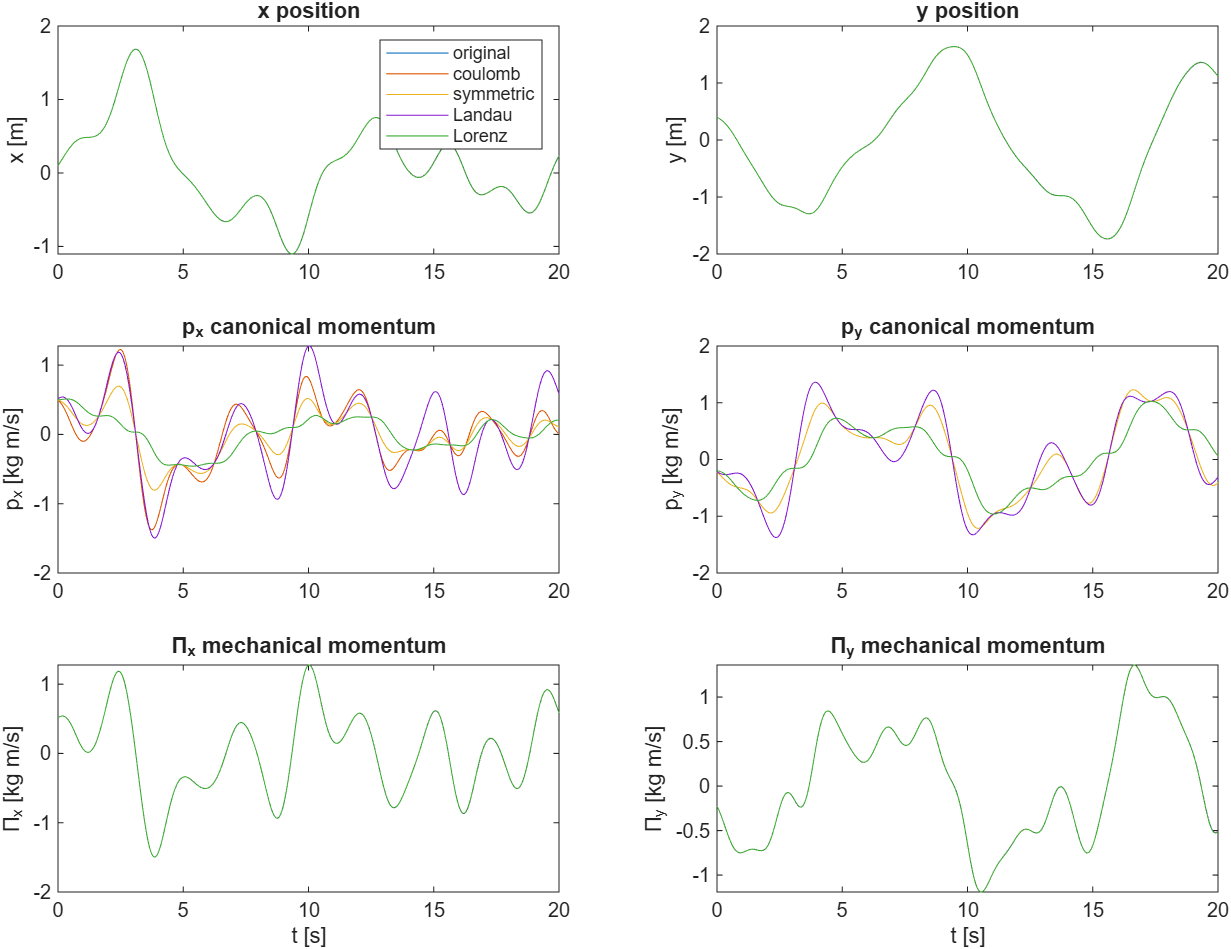}
    \caption{Trajectories, canonical momenta and mechanical momenta.}
    \hfill
    \label{fig:gauge1}
\end{subfigure}
\begin{subfigure}{\textwidth}
\centering
    \includegraphics[width=0.85\linewidth]{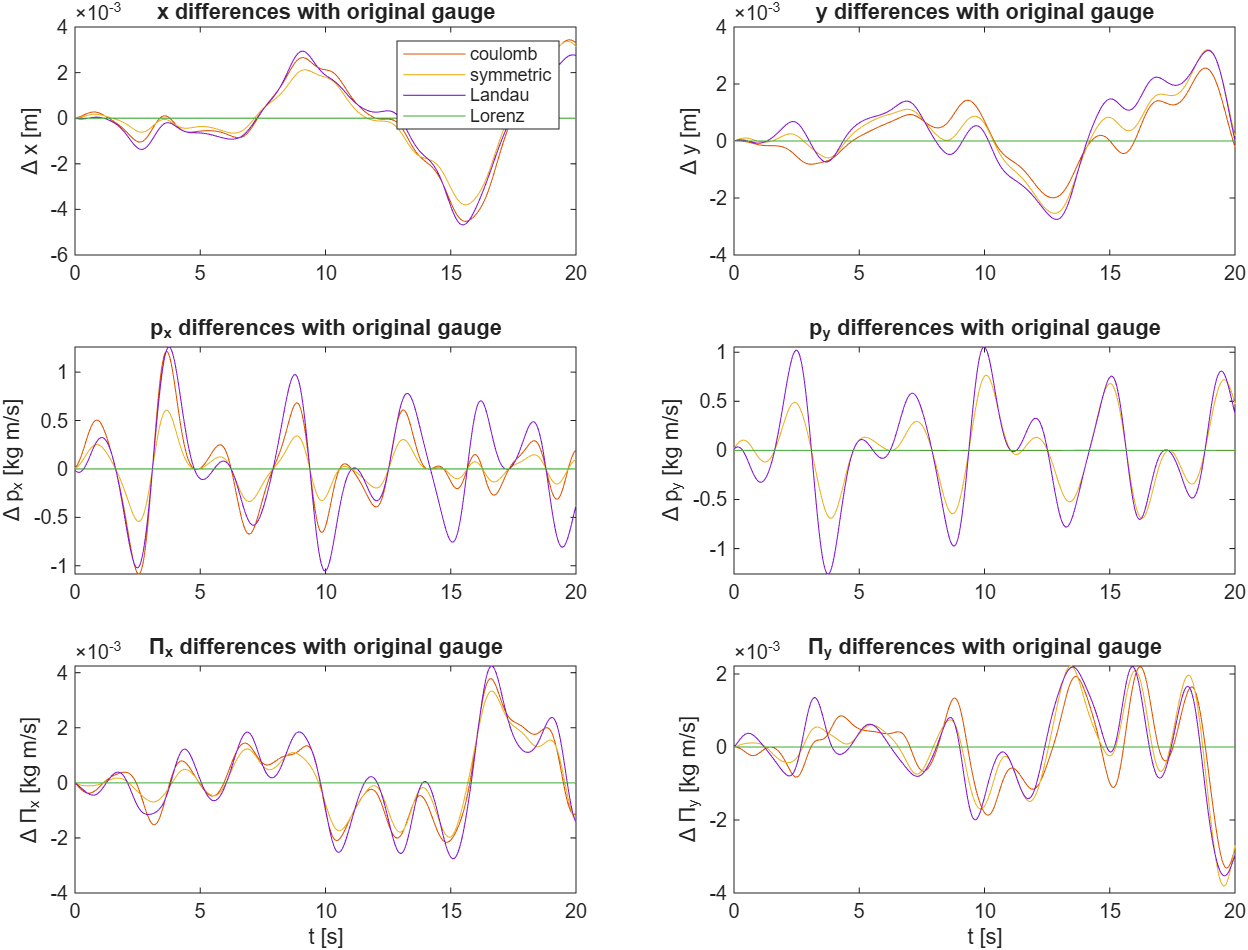}
    \caption{Differences with respect to the original gauge.}
    \label{fig:gauge2}
\end{subfigure}
\caption{$x$ and $y$ trajectories, canonical momenta $p_x$ and $p_y$, and mechanical momenta $\Pi_x$ and $\Pi_y$ for the different choices of gauges.}
\label{fig:gauge}
\end{figure}

\section{Conclusions and Future Work}\label{section: Conclusions and Future Work}

In this paper we have presented a geometric framework for discrete canonical transformations in the non-autonomous Hamiltonian setting.  
Starting from the (presymplectic) cosymplectic formulation on $T^*Q \times \mathbb{R}$, we introduced a generalized canonical transformation on the extended phase space $T^*(Q \times \mathbb{R})$ and showed that its projection acts as a Poisson morphism.  
This construction naturally induces a discrete flow that preserves the cosymplectic volume, the Poisson bracket $\{\cdot,\cdot\}_0$, and the symplectic form on each time fiber $T^*Q \times \{t\}$.  
The resulting integrator provides a geometric and structure-preserving alternative to traditional discretization schemes for non-autonomous systems.  
Numerical simulations confirm that the proposed method avoids artificial energy drift and accurately reproduces the geometric structure of the phase space over long integration times.

Moreover, by expressing the discrete map as the pullback of the canonical symplectic form on $T^*(Q \times \mathbb{R})$, we established a global, coordinate-free criterion for structure preservation that generalizes the local, matrix-based approach of Marthinsen and Owren~\cite{marthinsen2016geometric} and extends it to the cosymplectic setting.

Future work will first explore the application of the proposed integrator to the discretization of geometric filters and observers on Lie groups, particularly those based on the second-order minimum-energy framework developed in \cite{saccon2015second} and further explored in \cite{rigo2022second, rigo2023second, rigo2024state, trotti2025geometric}.
In that formulation, the underlying estimation problem is a \emph{time-dependent optimal control problem} on the tangent or cotangent bundle of a Lie group, where the Hamiltonian depends explicitly on time through the control and measurement processes.  
This makes the filtering dynamics inherently non-autonomous, and therefore naturally suited to the cosymplectic framework introduced in this paper.  
By discretizing the extended Hamiltonian system associated with the filtering problem, we aim to obtain discrete update maps that remain symplectic (or cosymplectic) on each time slice, thereby preserving group invariance, geometric consistency, and energy properties of the estimation dynamics.  
This approach would yield a new class of \emph{structure-preserving recursive filters on Lie groups}, combining the theoretical rigor of non-autonomous Hamiltonian discretization with the practical robustness required for real-time estimation.

Beyond this direct application, the discrete geometric framework developed here could also be extended toward \emph{variational formulations of estimation}, leading to \emph{discrete variational filters} derived from a discrete action principle.  
Such extensions would provide a unified variational–geometric approach to filtering and control on manifolds, potentially integrating estimation, prediction, and learning within the same symplectic–cosymplectic framework.


\bibliography{mybib}
\bibliographystyle{ieeetr}

\end{document}